\newif\ifshort
\ifshort
\documentclass[review,onefignum,onetabnum]{siamart190516}

\else
\documentclass[11pt,a4paper]{article}
\usepackage[margin=1in]{geometry}
\usepackage{geometry}
\usepackage{geometry}
\usepackage[dvips]{graphicx}
\usepackage{graphicx}
\usepackage{amsmath}
\usepackage{amsthm}
\usepackage{amssymb}
\usepackage{verbatim}
\usepackage{algorithmic}
\usepackage{algorithm}
\usepackage{enumerate}
\usepackage{framed}
\usepackage{marginnote}
\usepackage{color}
\newtheorem{theorem}{Theorem}[section]

\newtheorem{lemma}[theorem]{Lemma}

\newtheorem{proposition}[theorem]{Proposition}
\newtheorem{definition}{Definition}[section]
\fi

\usepackage{amsfonts}
\usepackage{tcolorbox}
\allowdisplaybreaks
\usepackage{hyperref}
\hypersetup{
pdftitle={Shown in AR File Information}, pdfstartview=FitH, 
}

\newcommand{\vaia}[1]
{%
\ifshort
\cite{mas19}%
\else
Section~\ref{#1} in appendix%
\fi
}%

\newtheorem{example}{Example}[section]

\newcommand{\GG}{g_1,\ldots,g_m}
\newcommand{\qd}{}
\newcommand{\SoS}{\text{\sc{SoS}}}
\newcommand{\sos}{\text{\sc{SoS}}}
\newcommand{\supp}{supp}
\newcommand{\y}{y^{N}}
\newcommand{\eps}{\varepsilon}

\newcommand{\F}{\mathcal{F}}
\newcommand{\FB}{\mathcal{F}_{01}}
\newcommand{\GB}{\text{Gr\"{o}bner }}
\newcommand{\FC}{FC}

\newcommand{\R}{\mathbb{R}}
\newcommand{\C}{\mathcal{C}}

\newcommand{\yy}[1]{y{\left [ #1 \right]}}
\newcommand{\PEXP}[1]{\tilde{\mathbb{E}}{\left [ #1 \right]}}
\newcommand{\set}[1]{\left\{ #1 \right\}}
\newcommand{\PS}{\mathcal{P}}
\newcommand{\conv}{\text{conv}}
\newcommand{\Ss}{\mathcal{S}}

\newcommand{\Ideal}[1]{{\mathbf{I}}\left( #1 \right)}
\newcommand{\IB}{{\mathbf{I}}_{01}}
\newcommand{\I}{{\mathbf{I}}}
\newcommand{\Pp}{\mathcal{P}}
\newcommand{\Field}{\mathbb{F}}

\newcommand{\dotp}[1]{\left\langle #1 \right\rangle}

\newcommand{\spn}[1]{\left\langle #1 \right\rangle}
\newcommand{\spns}{\left\langle \Ss \right\rangle}
\newcommand{\sap}{\Ss_A(\pi)}

\newcommand{\cone}{\mathbf{qmodule}}
\newcommand{\x}{x}
\newcommand{\Zz}{\mathbb{Z}}
\newcommand{\N}{\mathbb{N}}

\title{High Degree Sum of Squares Proofs, Bienstock-Zuckerberg Hierarchy  and
 Chv\'{a}tal-Gomory Cuts
\thanks{Preliminary version appeared in IPCO'17 \cite{Mastrolilli17}.
\ifshort
\funding{Supported by the Swiss National Science Foundation project 200020-169022 ``Lift and Project Methods for Machine Scheduling Through Theory and Experiments''.}
\else
Supported by the Swiss National Science Foundation project 200020-169022 ``Lift and Project Methods for Machine Scheduling Through Theory and Experiments''.
\fi
}}

\ifshort
\author{Monaldo Mastrolilli\thanks{IDSIA, Lugano, Switzerland. (\email{monaldo@idsia.ch}).}}
\else
\author{ Monaldo Mastrolilli\\
\small{IDSIA, Lugano, Switzerland.}\\
\small{ monaldo@idsia.ch}}
\fi

\begin{document}

\maketitle
\begin{abstract}
Chv\'{a}tal-Gomory (CG) cuts and the Bienstock-Zuckerberg hierarchy  capture useful linear programs that the standard bounded degree Sum-of-Squares ($\sos$) hierarchy  fails to capture.

In this paper we present a novel polynomial time $\sos$ hierarchy for 0/1 problems with a custom subspace of high degree polynomials (not the standard subspace of low-degree polynomials).
We show that the new $\sos$ hierarchy recovers the Bienstock-Zuckerberg hierarchy.
Our result implies a linear program that reproduces the Bienstock-Zuckerberg hierarchy as a polynomial-sized, efficiently  constructible extended formulation that satisfies all constant pitch inequalities. The construction is also very simple, and it is fully defined by giving the supporting polynomials.
Moreover, for a class of polytopes (e.g. set cover and packing problems), the resulting $\sos$ hierarchy optimizes in polynomial time over the polytope resulting from any constant rounds of CG-cuts, up to an arbitrarily small error in the solution value.

Arguably, this is the first example where different basis functions can be useful in \emph{asymmetric} situations to obtain a hierarchy of relaxations.
\end{abstract}

\section{Introduction}
The Lasserre/Sum-of-Squares ($\sos$) hierarchy \cite{Lasserre01,Nesterov2000,parrilo03,shor1987class} is a systematic procedure for constructing a sequence of increasingly tight semidefinite relaxations.
The $\sos$ hierarchy is parameterized by its \emph{level} (or \emph{degree}) $d$, such that the formulation gets tighter as $d$ increases, and a solution of accuracy $\eps>0$ can be found by solving a semidefinite program of size $(mn \log(1/\eps))^{O(d)}$, where $n$ is the number of variables and $m$ the number of constraints in the original problem. In this paper we consider 0/1 problems. In this setting, it is known that the hierarchy converges to the 0/1 polytope in $n$ levels and captures the convex relaxations used in the best available approximation algorithms for a wide variety of optimization problems
(see e.g. \cite{hierreading,blekherman2013semidefinite,Laurent03,Laurent2009} and the references therein).
%

In a recent paper Kurpisz, Lepp\"anen and the author~\cite{KurpiszLM17} characterize the set of 0/1 integer linear problems that still have an (arbitrarily large) integrality gap at level $n-1$. These problems are the ``hardest'' for the $\sos$ hierarchy in this sense. In another paper, the same authors~\cite{Kurpisz2017} consider a problem that is solvable in $O(n \log n)$ time and proved that the integrality gap of the $\sos$ hierarchy is unbounded at level $\Omega(\sqrt n)$ even after incorporating the objective function as a constraint (a classical trick that sometimes helps to improve the quality of the relaxation). All these $\sos$-hard instances are covering problems.

Chv\'{a}tal-Gomory (CG) rounding is a popular cut generating procedure that is often used in practice (see e.g.~\cite{Conforti:2014} and Section~\ref{sect:cg} for a short introduction).
There are several prominent examples of CG-cuts in polyhedral combinatorics,  including the odd-cycle inequalities of the stable set polytope, the blossom inequalities of the matching polytope, the simple M\"{o}bius ladder inequalities of the acyclic subdigraph polytope and the simple comb inequalities of the symmetric traveling salesman polytope, to name a few. Chv\'{a}tal-Gomory cuts capture useful and efficient linear programs that the bounded degree $\sos$ hierarchy fails to capture.
Indeed, the $\sos$-hard instances studied in \cite{KurpiszLM17} are the ``easiest'' for CG-cuts, in the sense that they are captured within the \emph{first} CG closure.
It is worth noting that it is NP-hard \cite{LetchfordPS11} to optimize a linear function over the first CG closure, an interesting contrast to lift-and-project hierarchies 
where one can optimize in polynomial time for any constant number of levels.\footnote{It has often been claimed in recent papers, that one can optimize over degree-$d$ $\sos$  via the Ellipsoid algorithm in $n^{O(d)}$ time.
In a recent work, O'Donnell~\cite{ODonnell17} observed that this often repeated claim is far from true. However,  this issue does not apply to most of the results published so far and to the applications of this paper. See also \cite{Mastrolilli19} for recent news.}

Interestingly, Bienstock and Zuckerberg \cite{BienstockZ06} proved that, in the case of set cover, one can separate over all CG-cuts to an arbitrary fixed precision in polynomial time. The result in \cite{BienstockZ06} is based on another result \cite{BienstockZ04} by the same authors, namely on a (positive semidefinite) lift-and-project operator (which we denote (BZ) herein) that is quite
different from the previously proposed operators.
This lift-and-project operator generates different variables for different relaxations. They showed that this flexibility can be very useful in attacking relaxations of some set cover problems.

These three methods, ($\sos$, CG, BZ), are to some extent incomparable, roughly meaning that there are instances where one succeeds while the other fails (see  \cite{AuT18} for a comparison between $\sos$ and BZ, the already cited \cite{KurpiszLM17} for ``easy'' cases for CG-cuts that are ``hard'' for $\sos$, and finally note that clique constraints are ``easy'' for $\sos$ but ``hard'' for CG-cuts~\cite{Pudlak97}, to name a few).

One can think of the standard Lasserre/$\sos$ hierarchy at level $d$ as optimizing an objective function over linear functionals that sends $n$-variate polynomials of degree at most $d$ (over $\R$) to real numbers.
The restriction to polynomials of degree $d$ is the standard way (as suggested in~\cite{Lasserre01,parrilo03} and used in most of the applications) to bound the complexity, implying a semidefinite program of size $n^{O(d)}$.
However, this is not strictly necessary for getting a polynomial time algorithm and it can be easily extended by considering more general subspaces having a ``small'' (i.e. polynomially bounded) set of basis functions (see e.g. Chapter 3 in \cite{blekherman2013semidefinite} and \cite{FawziSP15,Gatermann200495}). This is a less explored direction and it will play a key role in this paper. Indeed, the more general view of the $\sos$ approach has been used so far to exploit very symmetric situations (see e.g. \cite{FawziSP15,Gatermann200495,raymond2016symmetric}). For symmetric cases the use of a different basis functions has been proved to be very useful.

To the best of author's knowledge,
in this paper we give the first example where different basis functions can be useful in \emph{asymmetric} situations to obtain a hierarchy of relaxations. More precisely, we focus on 0/1 problems and show how to reframe the Bienstock-Zuckerberg hierarchy \cite{BienstockZ04} as an augmented version of the $\sos$ hierarchy that uses high degree polynomials (in Section~\ref{sect:sossetcov} we consider the set cover problem, that is the main known application of the BZ approach, and in Section \ref{sect:bz} we sketch the general framework that is based on the set cover case). 
The resulting high degree $\sos$ approach retains in one single unifying $\sos$ framework the best from the standard bounded degree $\sos$ hierarchy, incorporates the BZ approach and allows us to get, in polynomial time for any fixed $t\in \N$ and $\eps>0$, a solution that satisfies the $t$-th CG-closure and that is at most $\eps$-times worse than the optimal solution value
for both, set cover and packing problems (BZ guarantees this only for set cover problems).
Moreover, the proposed framework is very simple and, assuming a basic knowledge in $\sos$ machinery (see Section~\ref{sect:sosproofs}), it is fully defined by giving the supporting polynomials. This is in contrast to the Bienstock-Zuckerberg's hierarchy that requires an elaborate description \cite{BienstockZ04,zuckerberg2004set}. Finally, as observed in~\cite{AuT16} (see Propositions 25 and 26 in \cite{AuT16}), the performances of the Bienstock-Zuckerberg's hierarchy depend on the presence of redundant constraints.\footnote{I thank Levent Tun{\c{c}}el for pointing out his work to me~\cite{AuT16}.} The proposed approach removes these unwanted features.

We emphasize that one can also generalize the Sherali-Adams hierarchy/proof system in the same manner to obtain the covering results. We will give a detailed description of this in the following.
So the formulation that we are going to describe for the set cover problem is actually an explicit linear program, see Section~\ref{sect:BZ-LP},
that reproduces the Bienstock-Zuckerberg hierarchy as a polynomial-sized, efficiently constructible extended formulation that satisfies all constant pitch inequalities.

\paragraph{Paper Structure}
In order to make this article as self-contained as possible and accessible to non-expert readers, in Section~\ref{sect:sosproofs} we give a basic introduction to $\sos$-proofs/relaxations. However, we provide an introduction from a more general point of view, namely in terms of a generic subspace of polynomials. This is the ``non-standard'' flavour that will be advocated in this paper.

In Section~\ref{sect:simple_example} we consider a family of elementary Chv\'{a}tal-Gomory cuts that are ``hard'' for the standard Lasserre/$d$-$\sos$ relaxation. More precisely, for every $L$, we show that there exists $\eps>0$ such that the set $\{x\in[0,1]^n:\sum_{i=1}^{n}x_i\geq L+\eps\}$ has Lasserre rank at least $n-L$. On the other side, this can be easily fixed by using a different basis of high degree polynomials.

Our main application is given in Section~\ref{sect:sossetcov}, where we show that the $\sos$ framework
equipped with a suitably chosen polynomial-size spanning set of high degree polynomials, produces a relaxation, actually a compact linear program, for set cover problems for which all valid inequalities of a given, fixed pitch hold (Theorem~\ref{th:setcover}). The general BZ approach is discussed in Section~\ref{sect:bz}.

In Section~\ref{sect:cg}, we give the packing analog of Theorem~\ref{th:setcover}. In this case the standard $\sos$ hierarchy is sufficient. Moreover, we show that
the optimal value of maximizing a linear function over the $d$-th CG closure of a packing polytope (an NP-hard problem in general) can be approximated, to arbitrary precision and in polynomial time, by using the standard $\sos$ hierarchy.

Final remarks and future directions are given in Section~\ref{sect:conclusions}.
\ifshort
Due to page limit policy of 25 pages, the full version of the paper can be found in \cite{mas19}.
\else
\paragraph{Recent developments.}
Very recently Fiorini et al. \cite{FioriniHW17} claim a new approach to reproduce the Bienstock-Zuckerberg hierarchy. We remark that their framework is weaker than the one presented in this paper, meaning that does not generalize to packing problems (see Section~\ref{sect:cg}). Moreover, their proof is essentially based on similar arguments as used in this paper (formerly appeared in \cite{Mastrolilli17}). We give more details in the appendix.
\fi

\section{Sum of Squares Proofs and Relaxations}\label{sect:sosproofs}

In this section we give a brief introduction to $\sos$-proofs/relaxations. We refer to the monograph~\cite{Laurent2009} for an excellent in-depth overview.
We emphasize that there is no mathematical innovation in this section; all the details herein are basically known. However, instead of the ``standard'' $\sos$ description in terms of bounded degree monomials, we provide a definition as a function of a generic subspace of polynomials. This is used in the remainder of the paper.

We will use the following notation. Let  $\R[\x]:=\R[x_1, \dots , x_n]$ be the ring of polynomials over the reals in $n$ variables. Let $\R[x]_d$ denote the subspace of $\R[\x]$ of degree at most $d\in \N$. If $\Ss=\{s_1,\ldots,s_k\}$ is a set of polynomials in $\R[\x]$, then the \emph{span} of $\Ss$, denoted $\spn \Ss$, is the set of all linear combinations of the polynomials in $\Ss$, i.e. $\spn \Ss:=\{\sum_{i=1}^k c_i\cdot s_i:c_i\in\R \}$, and $\Ss$ is called the \emph{spanning set} of $\spn \Ss$.

The set $\F$ of feasible solutions of an optimization problem is usually described by a finite number of polynomial equations and/or inequalities. This is formalized by the following definition.
Let $\F\subset \R^n$ be defined as
\begin{align}\label{eq:F}
  \F&=\{x\in \R^n: f_i(x)=0 \ \forall i\in[\ell], g_j(x)\geq 0  \ \forall j\in[m]\},
\end{align}
where for each $i\in [\ell]$ and $j\in[m]$, $f_i(x),g_j(x)\in \R[x]$ and where $[\ell]$ denote $\{1,2,\ldots, \ell\}$. Here, $\F$ is called a \emph{basic closed semialgebraic set}.
For the sake of brevity, throughout this document, while referring to a {semialgebraic set}, we implicitly assume a \emph{basic closed} semialgebraic set.

One could write many other constraints that are equally valid on the set $\F$.
For example, we are able to produce further polynomials vanishing on the set $\F$ by considering linear combinations of $f_i(x)$ with polynomial coefficients.
The set of all polynomials generated this way is a polynomial ideal.

\begin{definition}\label{def:ideal}
The \emph{{ideal}} generated by a finite set $\{f_1,\ldots, f_\ell\}$ of polynomials in $\R[x]$ is defined as
$$\Ideal{ f_1,\ldots,f_\ell}:= \left\{\sum_{i=1}^\ell t_i \cdot f_i\ :\ t_1,\ldots,t_\ell\in \R[x]\right\}.$$
\end{definition}
A polynomial $p\in\R[x]$ is a \emph{sum of squares} ($\sos$) if it can be written as the sum of squares of some other polynomials.
If these last polynomials belong to a subspace $\spns\subseteq \R[x]$, for a given spanning set $\Ss\subseteq \R[\x]$, then we say that $p$ is $\Ss$-$\sos$.
\begin{definition}\label{def:polysos}
For $\Ss\subseteq \R[x]$, a polynomial $p\in\R[x]$ is $\Ss$-$\sos$ if $p\in \Sigma_\Ss$ where
$$\Sigma_\Ss:=\{p\in\R[\x]: p=\sum_{i=1}^{r} q_i^2, \text{ for some } r\in \N \text{ and } q_1,\ldots, q_r\in \spns\}.$$
\end{definition}
As for the vanishing polynomials on $\F$, we are able to produce further valid inequalities for set $\F$ by multiplying $g_j(x)$ against $\sos$ polynomials, or by taking conic combinations of valid constraints. This gives the notion of {quadratic module}.
\begin{definition}\label{def:qmodule}
For $\Ss\subseteq \R[x]$, the $\Ss$-\emph{quadratic module} generated by a finite set $\{g_1,\ldots, g_{m}\}$ of polynomials in $\R[x]$ is defined as
$$\cone_\Ss( g_1,\ldots,g_{m}):=\left\{s_0+ \sum_{i=1}^{m} s_i \cdot g_i: s_0,s_1,\ldots,s_m\in\Sigma_\Ss \right\}.$$
\end{definition}
Certifying that a polynomial $p\in \R[x]$ is non-negative over a semialgebraic set~$\F$ is an important problem in optimization, as certificates of non-negativity can often be leveraged into optimization algorithms.
For example let $p:=p'-\gamma$, where $p'\in \R[x]$ and $\gamma$ is a real number. If we can certify that $p$ is non-negative over $\F$ then the infimum of $p'$ is not smaller than~$\gamma$. We will elaborate more on this in Section~\ref{sect:dualsos}.
\begin{definition}
For $\Ss\subseteq \R[x]$ and $p(x)\in\R[x]$, a $\Ss$-$\sos$ certificate of non-negativity of $p(x)$ over $\F$ (see~\eqref{eq:F}) is given by a polynomial identity of the form
\begin{align}\label{eq:sos_certificate}
  p(x) &= f(x)+g(x),
\end{align}
for some $f(x)\in \Ideal{ f_1,\ldots,f_\ell}$  and $g(x)\in \cone_\Ss( g_1,\ldots,g_{m})$.
\end{definition}
Notice that for all $x\in \F$ the right-hand side of~\eqref{eq:sos_certificate} is manifestly non-negative, thereby certifying that $p(x)\geq 0$ over $\F$.

In the following, whenever $\Ss=\R[x]$, we drop $\Ss$ from the notation. So $\Sigma$, $\sos$ and $\cone( g_1,\ldots,g_{m})$ denote $\Sigma_{\R[x]}$, ${\R[x]}$-$\sos$ and $\cone_{\R[x]}( g_1,\ldots,g_{m})$, respectively.

A natural question arises: Can all valid constraints be generated this way? Unless further assumptions are made, the answer is negative (see, e.g.~\cite{blekherman2013semidefinite}).
However, for the applications of this paper, we are interested in the case $\F$ is the set of feasible solutions of a 0/1 integer linear program, with $n$ variables and $m$ linear constraints:
 \begin{align}\label{eq:semialgebraicset}
\FB&:=\{x\in \R^n: x_i^2-x_i=0  \ \forall i\in[n],\ g_j(x)\geq 0  \ \forall j\in[m]\},
\end{align}
where $x_i^2-x_i=0$ encodes $x_i\in\{0,1\}$ and each constraint $g_j(x)\geq0$ is linear. Under this assumption the answer of the above question is positive, as shown in the following.
(Actually, the linearity of the constraints is not necessary for this purpose.)
We review this derivation from a slightly different perspective,
by highlighting several aspects that will play a role in our proofs.

We start with some preliminaries.
%
The set of polynomials in $\R[x]$ that vanish on the Boolean hypercube $\Zz^n_2$ is the ideal $$\IB:=\Ideal{x_1^2-x_1,\ldots,x_n^2-x_n}.$$
\begin{definition}
Let $\I$ be an ideal, and let $f,g\in \R[\x]$. We say that $f$ and $g$ are \emph{congruent modulo $\I$}, written $f \equiv g \pmod \I$, if $f-g\in \I$.
\end{definition}

From the above definition, a $\Ss$-$\sos$ certificate of non-negativity of $p(x)$ over $\FB$ is given by a polynomial congruence of the form
\begin{align}\label{soscertificate01}
  p(x) &\equiv g(x) \pmod \IB,
\end{align}
for some $g(x)\in \cone_\Ss( g_1,\ldots,g_{m})$. For the sake of brevity, whenever we use ``$\equiv$'' we assume that the congruence is modulo $\IB$ (unless differently specified).

Let us introduce an indicator multilinear polynomial that will play an important role throughout this paper.
For $I\subseteq Z\subseteq [n]$, the \emph{Kronecker delta} polynomial is defined as:
\begin{equation}\label{eq:dirac}
\delta^Z_I:=\prod_{i\in I} x_i \prod_{j\in Z\setminus I} (1-x_j).
\end{equation}
If $Z=\emptyset$ we assume that $\delta^Z_I=1$.
Let $x^Z_I$ denote  the 0/1 (partial) assignment with $x_i=1$ for $i\in I$, and $x_j=0$ for $j\in Z\setminus I$.
Notice that $\delta^Z_I$ is an indicator polynomial that is 1 when
its variables get assigned values according to $x^Z_I$.
%
Moreover, the following identities hold:
\begin{align}
  &\sum_{I\subseteq Z} \delta^Z_I=1, \label{Kr1} \\
  &\left(\delta^Z_I\right)^2\equiv \delta^Z_I, \label{Kr2}\\
  &\delta^Z_I\delta^Z_J\equiv 0,\text{ for }I,J\subseteq Z\text{ with }I\not=J. \label{Kr3}
\end{align}
By using \eqref{Kr2} and \eqref{Kr3} we have (for $Z\subseteq [n]$ and $W\subseteq 2^Z$)
\begin{align}\label{Kr4}
  \left(\sum_{I\in W} \delta^Z_I\right)^2&\equiv \sum_{I\in  W} \delta^Z_I.
\end{align}
For any given $p(x)\in \R[x]$, let us use $p(x^Z_I)$ to denote $p(x)$ after the partial assignment defined by $x^Z_I$: for example if $p(x)=p_0+\sum_{i=1}^n p_i \cdot x_i$ then $p(x^Z_I)= p_0+\sum_{i\in I} p_i+\sum_{[n]\setminus Z}p_i\cdot x_i$. Then the following holds:
\begin{align}\label{Kr5}
 \delta^Z_Ip(x) & \equiv \delta^Z_Ip\left(x^Z_I\right).
\end{align}
These basic facts will be used several times.

\paragraph{$\sos$ Proofs Over the Boolean Hypercube}
For any given polynomial $p(x)\in \R[\x]$ that is non-negative over $\FB$, we are interested in certifying this property by exhibiting a $\sos$ certificate. 
With this aim, partition the Boolean hypercube into two sets $N^+:=\left\{I\subseteq [n]:p\left(x_I^{[n]}\right)\geq 0\right\}$ and $N^-:=\left\{I\subseteq [n]:p\left(x_I^{[n]}\right)< 0\right\}$. If $p(x)$ is non-negative over $\FB$, then for each $I\in N^-$ there exists a constraint that is violated on $x_I^{[n]}$, i.e. there is a mapping $h:N^-\rightarrow [m]$ such that $g_{h(I)}\left(x_I^{[n]}\right)<0$. To ease the notation, we drop the exponent ``$[n]$'' from $x_I^{[n]}$ and $\delta_I^{[n]}$.
Then:
\begin{align}\label{eq:interpol}
\begin{split}
&p(x)=\overbrace{\left(\sum_{I\subseteq [n]} \delta_I\right)}^{=1\text{ by \eqref{Kr1}}}p(x)\overset{\text{by \eqref{Kr5}}}{\equiv} \sum_{I\in N^+} \delta_I p(x_I)+\sum_{I\in N^-} \delta_I \frac{p(x_I)}{g_{h(I)}(x_I)}g_{h(I)}(x_I)
 \\
&\overset{\text{by \eqref{Kr4} and \eqref{Kr5}}}{\equiv} \underbrace{\left(\sum_{I\in N^+} \delta_I \sqrt{p(x_I)}\right)^2}_{s_0}+\sum_{I\in N^-} \underbrace{\left(\delta_I \sqrt{\frac{p(x_I)}{g_{h(I)}(x_I)}}\right)^2}_{s_{h(I)}}g_{h(I)}(x).  
\end{split}
\end{align}
It follows that any non-negative polynomial over $\FB$ admits a $\Ss$-$\sos$ certificate where $\Ss$ is the set $\{\delta_I:I\subseteq [n]\}$ of Kronecker delta multilinear polynomials.
The quotient ring $\R[\x]/\IB$ is the set of equivalence classes for congruence modulo~$\IB$. Polynomials from the quotient ring $\R[\x]/\IB$ are in bijection with square-free (also known as multilinear) polynomials in $\R[\x]$. We will use $\R[\x]/\IB$ to denote the subspace of multilinear polynomials. The aforementioned Kronecker delta polynomials form a basis for the subspace of multilinear polynomials $\R[\x]/\IB$.
The next proposition summarizes the above.
%
%
%
\begin{proposition}\label{th:sos}
Let $\spns=\R[\x]/\IB$.  If $p(x)\in \R[x]$ is non-negative over $\FB$ then it admits a  $\Ss$-$\sos$ certificate of the form
\begin{align}\label{eq:01sos_certificate}
  p(x) &\equiv g(x) \pmod \IB,
\end{align}
for some $g(x)\in \cone_\Ss( g_1,\ldots,g_{m})$.
\end{proposition}

The existence of a $\Ss$-$\sos$ certificate can be decided by solving a semidefinite programming (SDP) feasibility problem whose matrix dimension is bounded by $O(|\Ss|)$.  We refer to~\cite{blekherman2013semidefinite,FawziSP15} and \vaia{sect:qsosproofscomplexity} for details and an example.

If $\spns=\R[\x]/\IB$, then the SDP has exponential size. The ``standard'', namely the ``most used'' way to bound the complexity is to restrict the spanning set $\Ss$ of $\Ss$-$\sos$ certificates to be the standard monomial basis of constant degree $d=O(1)$.
This bounds the degrees of the polynomials in $\Ss$-$\sos$ certificates to be a constant, and a non-negativity certificate is computed by solving a semidefinite program of size $n^{O(d)}$. Clearly this restriction imposes severe limitations on the kind of proofs that can be obtained. This type of approach was first proposed by Shor~\cite{shor1987class}, and the idea was taken further by Parrilo~\cite{parrilo03} and Lasserre \cite{Lasserre01}.

However, this modus operandi with bounded degree monomials can be extended to other subspaces $\spns$ having ``small'' spanning sets $\Ss$, i.e. with $|\Ss|=n^{O(d)}$ for some $d=O(1)$. This is a less explored direction and it will play a key role in this paper.
%
%
%
%
\subsection{0/1 Optimization and $\sos$ Relaxations}\label{sect:dualsos}

As already remarked, a number $\gamma$ is a global lower bound of a polynomial $p(x)$  over $\FB$ if and only if $p(x)-\gamma$ is non-negative over $\FB$.
For 0/1 problems, without loss of generality, we can assume that $p(x)$ is in multilinear form and therefore $p(x)\in \R[x]_n$.
For $\spns \subseteq \R[\x]/\IB$,
a relaxation of the above optimization problem is obtained by computing the $\sup \gamma$ such that $p(x)-\gamma$ has a $\Ss$-$\sos$ certificate of nonnegativity:
\begin{align}\label{eq:primal}
\begin{split}
\sup_{\gamma} \{&\gamma: p(x)-\gamma\in \C_{\Ss}\},
\end{split}
\end{align}
where
\begin{align}\label{eq:primal_cone}
  \C_{\Ss}&:=\{q+r: q\in\cone_{\Ss}(\GG), r\in\IB\cap \R[n]_{2n}\}
\end{align}
is a set of $\Ss$-$\sos$ certificates.
Note that \eqref{eq:primal} is indeed an approximation, since it could be that $p(x)-\gamma$ is non-negative
for some $\gamma$, but the set $\Ss$ is ``too small'' so that a $\Ss$-$\sos$ certificate does not exist.
However, enlarging $\Ss$ increases the number of possible certificates and thus tightens the approximation.
For 0/1 semialgebraic sets and multilinear $p(x)$, we can always reduce to the case where the polynomials of $\sos$ certificates have degree at most $2n$, since for $\spns = \R[\x]/\IB$, the relaxation \eqref{eq:primal} is actually exact, as shown in~\eqref{eq:interpol}. This explains why we can restrict $r\in\IB\cap \R[n]_{2n}$ in \eqref{eq:primal_cone}.

\subsection{Duality and the Lasserre/$\sos$ Hierarchy}\label{sect:dual_lasserre}
The linear space of all real polynomials of $n$ variables and degree at most $d$ is isomorphic to the Euclidean space $\R^{n+d \choose{d}}$. Indeed, a simple combinatorial argument shows that any degree-$d$ polynomial $p(x)$ can have at most ${n+d \choose{d}}$ monomials, which we can order in some arbitrary way (ordered basis). Then, we can put the coefficients in a column vector $p$, in the selected order, and thus obtain a bijective mapping to $\R^{n+d \choose{d}}$. We will say that $p$ is the \emph{column vector representation} of $p(x)$ in the (ordered) \emph{standard monomial basis}. Then, for $\Ss \subseteq \R[\x]/\IB$, set $\C_{\Ss}$ (see \eqref{eq:primal_cone}) is (isomorphic to) a subset of $\R^{n+2n \choose{2n}}$, and it can be shown to form a \emph{cone} in the sense of convex geometry.

We emphasize that in the above arguments we have chosen the standard monomials as basis for the column vector representation of polynomials. It is clear that other bases are possible.
Actually, for our main application we will use a different basis. More generally, any linear space $V$ (of polynomials) is isomorphic to the space of column vectors of a certain dimension: choose an ordered basis $b^\top=(b_1,\ldots,b_k)$ for the linear space $V$ (of polynomials), the column vector representation of $p(x)\in V$ is a vector $p\in \R^k$ such that $p(x)=b^\top p$.

\paragraph{Dual Program}
Recall that in linear algebra, a \emph{linear functional} $y$ is a linear map from a linear space $V$ to its field $\Field$ of scalars. A linear functional $y$ is a \emph{linear} function:
\begin{align}\label{eq:func_linearity}
  y(\alpha\cdot v + \beta \cdot w)&=\alpha \cdot y(v) + \beta \cdot y(w)\quad \forall v,w\in V, \  \forall \alpha,\beta\in \Field.
\end{align}
In $\R^k$, for $k\in \N$, linear functionals are represented as vectors and their action on vectors is given by the inner product:
let $y,z\in \R^k$, the evaluation of $y$ at $z$ is denoted by the inner product $\dotp{z,y}$, that is $\dotp{z,y}=y(z)$.
Let $\C$ be a set in $\R^k$ equipped with an inner product $\dotp{z,y}=y(z)$. The \emph{dual} cone\footnote{Recall, in finite dimension, \emph{topological} and \emph{algebraic} duals are the same.} of $\C$ is defined by
\begin{align}\label{dual_cone}
  \C^*=\{y\in \R^k: y(z)\geq 0 \ \forall z\in \C \}.
\end{align}
In other words, the dual cone is the set of linear functionals that are non-negative on the primal cone.
Consider a standard conic program over a cone $\C$ and its dual:
\begin{align} \label{eq:p}
{\bf{Primal}}&:\sup_z \{\dotp{c,z}: p - Az \in \C\}; \quad
{\bf{Dual}}:\inf_y\{\dotp{p,y}: A^\top y = c;y\in \C^*\}.
\end{align}

To find the dual program of~\eqref{eq:primal} as a conic optimization problem, choose an (ordered) \emph{basis} for the polynomials in $\C_{\Ss}$ (we will say a little bit more about this later). The dimension of this basis defines the dimension of the linear functionals $y$: there is one entry in $y$ for each polynomial in the basis. Set $p$ in~\eqref{eq:p} to be the column vector representation of polynomial $p(x)$ in~\eqref{eq:primal} according to the chosen (ordered) basis.  Consider representing the variable $\gamma$ as the constant term of a polynomial $z(x)$. Let $z$ be the column vector representation of $z(x)$ and maximize its inner product with a suitably chosen vector $c$ so that $\dotp{c, z}=\gamma$.

 For the standard monomial basis choose $c=(1,0,\ldots,0)^\top$ and the matrix $A$ such that $A_{0,0}=1$ and $A_{i,j}=0$ elsewhere. 
So under this choice, we get as the Dual:
\begin{align}\label{eq:dual}
\begin{split}
 \inf_y\{\dotp{p,y}: y_0=1;\  y \in \C^*_{\Ss}\}.
\end{split}
\end{align}

The dual cone $\C^*_{\Ss}$ of $\C_{\Ss}$ turns out to have some nice properties, as explained below.
For any given polynomial $p(x)\in\R[x]$, we will use $\yy{p(x)}$ to denote $y(p)$ (or $\dotp{p,y}$), where $p$ is the column vector representation of $p(x)$ according to the chosen (ordered) basis, and $y$ is a linear functional.
With respect to any chosen vector basis for the polynomials from $\C_{\Ss}$,
the elements of the dual space $\C^*_{\Ss}$ define linear functionals $\yy{\cdot}$ (sometimes called \emph{pseudo-expectation} functionals and denoted with $\PEXP{\cdot}$) on polynomials that satisfy:
%

\begin{enumerate}
  \item (\textbf{Normalization}) $\yy{1}=1$\label{eq:d1};
  \item (\textbf{Linearity}) $\yy{\alpha \cdot p(x)+\beta \cdot q(x)}=\alpha \cdot \yy{p(x)}+\beta \cdot \yy{q(x)}$, for all $p(x),q(x) \in \C_{\Ss}$ and $\alpha,\beta\in \R$;\label{eq:linearity}
  \item (\textbf{Positivity}) $\yy{q(x)^2} \geq 0$, for all $q(x) \in {\spns}$;\label{eq:d2}
  \item (\textbf{Positivity}) $\yy{q(x)^2\cdot  g_i(x)} \geq 0$, for all $q(x) \in {\spns}$, for all $i\in[m]$;\label{eq:d3}

  \item (\textbf{Multilinearity}) $\yy{t(x)\cdot(x_i^2-x_i)} = 0$, for all $t(x)\in\R[x]$, for all $i\in[n]$.\label{eq:d4}
\end{enumerate}

Condition \eqref{eq:d1} says that the constant polynomial 1 is mapped to 1.
Note that in~\eqref{eq:dual}, $y_0=1$ comes directly from \eqref{eq:d1} (in the standard monomial basis we have $\yy{1}=\dotp{(1,0,\ldots,0)^\top,y}=y_0$).

Condition \eqref{eq:linearity} follows from the linearity of linear functionals (see \eqref{eq:func_linearity}).
Note that assigning arbitrary values to the entries of the linear functional $y$ guarantees linearity. Indeed, the entries of $y$ are linearly independent because they correspond to the ``linearization''
of the polynomials that form a basis for $\C_{\Ss}$, which are linearly independent. This is the only place where we need linear independence. Alternatively, we can choose a spanning set of polynomials for $\C_{\Ss}$ and impose the linearity condition~\eqref{eq:linearity}.

%

Conditions \eqref{eq:d2}, \eqref{eq:d3} and \eqref{eq:d4} follow from the definition of the dual cone (see \eqref{dual_cone}) of $\C_{\Ss}$ (see \eqref{eq:primal_cone}). Note that the multilinearity condition \eqref{eq:d4} can be easily enforced by restricting to multilinear polynomials: any given polynomial $p(x)$ will be replaced by its \emph{multilinear form}, denoted $\overline{p(x)}$, i.e. the normal form after polynomial division by the \GB basis $\{x_i^2-x_i:i\in [n]\}$.\footnote{The multilinear form of $p(x)$ is obtained by replacing every occurrence in $p(x)$ of ``$x_i^k$'' with ``$x_i$'', whenever $i\in[n]$ and $k\geq 2$; for example $x_1\cdot x_2 +2\cdot x_2$ is the multilinear form of $x_1^3\cdot x_2 +2\cdot x_2^2$.} So in conditions \eqref{eq:d2} and \eqref{eq:d3}, we replace $q(x)^2$ and $q(x)^2\cdot  g_i(x)$, with their multilinear forms $\overline{q(x)^2}$ and $\overline{q(x)^2\cdot  g_i(x)}$, respectively.
From now on, we will restrict to the subspace of multilinear polynomials $\R[\x]/\IB$. This allows us to enforce the multilinearity condition \eqref{eq:d4}.

By the above arguments, we can restrict to the polynomials from $\C_{\Ss}$ that are multilinear. These polynomials are spanned by the following set of multilinear polynomials $$T:=\{\overline{x_i\cdot p\cdot q}: p,q\in \Ss, i\in[n]\cup\{0\}\},$$ where $x_0:=1$, and recall that we are considering linear constraints $g_i(x)\geq 0$, for $i\in [m]$. So, the vector $y$ has one entry for each polynomial that belongs to a chosen basis for the span $\langle T \rangle$ of $T$.
By assuming this, we can
reformulate~\eqref{eq:dual}, with respect to a chosen basis for $\langle T \rangle$, as
\begin{align}
\inf\  & \yy{p} \label{eq:obj}\\
s.t.\  & \yy{1}=1;\label{eq:l1}\\
& \yy{\overline{q(x)^2}} \geq 0, \quad \forall q(x) \in {\spns}; \label{eq:psdcond}\\
& \yy{\overline{q(x)^2\cdot  g_i(x)}} \geq 0, \quad \forall q(x) \in {\spns}, \forall i\in[m].  \label{eq:2psdcond}
 \end{align}

The program \eqref{eq:obj}-\eqref{eq:2psdcond} is actually a \emph{semidefinite program} whose matrix dimension is bounded by $O(|\Ss|)$, that we call \emph{$\Ss$-$\sos$ relaxation}.
%
%
To see this, let $b^\top=(b_1,\ldots,b_k)$ be an (ordered) basis for $\spns$, for some $k\leq |\Ss|$. Then, consider any polynomial $q(x)\in \spns$, and let $q$ be its column vector representation according to the (ordered) basis $b^\top$, i.e. $q(x)=b^\top q$. Then $q(x)^2= \langle q q^\top, b b^\top \rangle$. Let $M(y)$ be a $|b|\times |b|$ square matrix indexed by the pairs $(b_i,b_j)\in b\times b$, such that the $(b_i,b_j)$-th entry of $M(y)$ is equal to $\yy{\overline{b_i b_j}}$.
Recall that $\yy{q(x)^2}$ is equal to $\langle y, p \rangle$, where $p$ is the column vector representation of $q(x)^2$.
By simple inspection note that $\yy{\overline{q(x)^2}}=\langle qq^\top, M(y)\rangle$. It follows that Condition~\eqref{eq:psdcond} is equivalent to impose $\langle qq^\top, M(y)\rangle\geq 0$ for all $q$, which is equivalent to require $M(y)$ to be positive semidefinite. A similar argument holds for Condition~\eqref{eq:2psdcond}.

\paragraph{Standard and Generalized $\sos$ Relaxations} When $\Ss$ is the standard (multilinear) monomial basis of degree $\leq d$, then the $\Ss$-$\sos$ relaxation~\eqref{eq:obj}-\eqref{eq:2psdcond} is the (standard) \emph{Lasserre/$\sos$-hierarchy} parameterized by the degree $d\in \N$, in short, denoted by $d$-$\sos$. $\Ss$-$\sos$ generalizes $d$-$\sos$ relaxations by working with a generic set $\Ss$ of polynomials. In this case, the aforementioned matrix $M(y)$ is the so-called (truncated) \emph{ moment matrix}.

Note that in standard $\sos$, set $T$ forms a basis for $\langle T \rangle$, and it is the set of all (multilinear) monomials of degree at most $2d+1$. The variables in $d$-$\sos$ are the entries of the linear functionals $y$, which correspond to the ``linearization'' of the polynomials from $T$.
\paragraph{Standard and Generalized Sherali-Adams Relaxations}
If $\Ss$ is again the standard monomial basis of degree $\leq d$ \emph{and} we further relax \eqref{eq:psdcond} and \eqref{eq:2psdcond} by
\begin{align}
& \yy{\overline{q(x)}} \geq 0, \quad \forall q(x) \in \Ss, \label{eq:SAcond1}\\
& \yy{\overline{q(x)\cdot g_i(x)}} \geq 0, \quad \forall q(x)\in \Ss, \forall i\in[m],  \label{eq:SAcond2}
 \end{align}
then we obtain the so called \emph{Sherali-Adams (SA) hierarchy} of relaxations, denoted $d$-SA, and defined by~\eqref{eq:obj}, \eqref{eq:l1}, \eqref{eq:SAcond1}, \eqref{eq:SAcond2}. This is again parameterized by $d$, but it is a \emph{linear program} (this follows from \eqref{eq:SAcond1}, \eqref{eq:SAcond2} and the definition of linear functionals where their action on vectors is given by the dot product) of size $O(|\Ss|)=n^{O(d)}$. Note that both hierarchies, $d$-$\sos$ and $d$-SA, have the same spanning set $\Ss$ of monomials, which are \emph{non-negative} over the Boolean hypercube.

In the definition of $d$-SA we restrict to work with polynomials from $\overline{q(x)\cdot g_i(x)}$ for $q(x)\in \Ss$, and $i\in[m]$. Let $$T_{SA}:=\{\overline{x_i\cdot p}: p\in \Ss, i\in[n]\cup\{0\}\}.$$ When $\Ss$ is the standard monomial basis then $T_{SA}$ is a basis for $\langle T_{SA} \rangle$, and it is the set of all (multilinear) monomials of degree at most $d+1$. It follows that the variables in $d$-SA are the entries of the linear functionals $y$, which correspond to the ``linearization'' of the polynomials from $T_{SA}$.

We generalize $d$-SA relaxations to work with a generic set $\Ss$ of polynomials (that is non-negative over the Boolean hypercube), and obtain $\Ss$-SA. The relaxation $\Ss$-SA is a linear program with $O(|\Ss|)$ linear constraints, which correspond to \eqref{eq:l1}, \eqref{eq:SAcond1} and \eqref{eq:SAcond2}.
%


%
 We conclude our overview on
 {$\sos$-relax-ations} by pointing out the following fact.

 \begin{proposition}\label{th:pdrel}
 If $p(x)$ admits a $\Ss$-$\sos$ certificate of non-negativity over $\FB$, then $\yy{p(x)}\geq 0$ holds for the corresponding $\Ss$-$\sos(\FB)$ relaxation~\eqref{eq:l1}-\eqref{eq:2psdcond}.
 \end{proposition}
 \begin{proof}
   By assumption, for some $f(x)\in \IB$ and $g(x)\in \cone_{\Ss}(\GG)$, we have $p(x)=f(x)+g(x)$. Then, $\yy{p(x)}= \yy{f(x)]+y[g(x)}=0+\yy{s_0}+ \sum_{i=1}^{m} \yy{s_i \cdot g_i}$ for some $s_0,s_1,\ldots,s_m\in\Sigma_\Ss$. By \eqref{eq:psdcond} and \eqref{eq:2psdcond}, each addend of the sum is non-negative and we have $\yy{p(x)}\geq 0$.
 \end{proof}

 By Proposition~\ref{th:pdrel}, if $p(x):=\sum_i a_i x_i -a_0\geq 0$ is a valid linear inequality for all $x\in\FB$ that admits a $\Ss\text{-}\sos$ certificate, then $\yy{p(x)}=\sum_i a_i \yy{x_i} -a_0\geq 0$. Note that $\{\yy{x_1},\ldots,\yy{x_n}\}$ is the solution $y$ of~\eqref{eq:l1}-\eqref{eq:2psdcond} projected to the original space of the variables. So, the (projected) solution of the $\Ss$-$\sos$ relaxation~\eqref{eq:l1}-\eqref{eq:2psdcond} satisfies $p(x)\geq 0$.
This implies the following informal ``recipe'' that  we will follow in the remainder of the paper. (Similar arguments hold for $\Ss$-SA.)
\medskip

\begin{tcolorbox}[colframe=white]\label{recipe}
\textbf{Recipe}: Assume that we are looking for a ``small'' relaxation for $\FB$ that satisfies a potentially ``large'' set of linear constraints $Ax\geq b$ that are valid for all $x\in \FB$. With this aim, search for a ``small'' spanning set $\Ss\subseteq \R[x]$ (if one exists) such that $Ax- b$ admits a $\Ss\text{-}\sos$ certificate. If we succeed, then the corresponding $\Ss$-$\sos$ relaxation~\eqref{eq:l1}-\eqref{eq:2psdcond} satisfies our goal.
\end{tcolorbox}

\section{A Simple Chv\'{a}tal-Gomory Cut That is Hard for $d$-$\sos$}\label{sect:simple_example}

For illustrative purposes, in this section we consider a simple example where the standard Lasserre/$d$-$\sos$ relaxation provably fails for ``large'' $d$. However, this can be easily fixed by using $\Ss$-$\sos$ with a ``small'' spanning set $\Ss$ of high degree polynomials.

The example is motivated by the following situation.
Consider the rational polyhedra $P=\{x\in \R^n: Ax\geq b\}$ with $A\in \Zz^{m\times n}$ and $b\in \Zz^{m}$.
Inequalities of the form $(\lambda^\top A)x\geq \lceil \lambda^\top b\rceil$, with $\lambda\in \R_+^m$, $\lambda^\top A\in \Zz^n$, and $\lambda^\top b\not\in \Zz$ are commonly referred to Chv\'{a}tal-Gomory (CG) cuts (further information on CG-cuts are provided in Section~\ref{sect:cg}).
It is a natural question to study how many levels (or degree $d$) of the {``standard''} Sum-of-Squares hierarchy, i.e. $d$-$\sos$, 
 are necessary to strengthening $(\lambda^\top A)x\geq  \lambda^\top b$ to get $(\lambda^\top A)x\geq \lceil \lambda^\top b\rceil$.
With this aim, consider the following semialgebraic set:
\begin{align}\label{symknapineq}
\FB=\{x\in \R^n: x_i^2-x_i=0  \ \forall i\in [n], \sum_{i=1}^n x_i\geq b\},
\end{align}
where $b\in \mathbb{Q}_+$ is intended to be a positive fractional number. Obviously, any feasible integral solution satisfies $\sum_{i=1}^n x_i\geq \lceil b \rceil$, and this is promptly captured by the first CG closure.

The following Theorem~\ref{th:maxsymconstr} (the proof can be found in \vaia{sect:toyd})  shows that regardless of whether $b$ is ``small'', i.e. $b=O(1)$, or ``large'', i.e. $b=\Omega(n)$,
$d$-$\sos(\FB)$ fails to enforce the simple CG-cut when $d= o(n)$.
%
\begin{theorem}\label{th:maxsymconstr}
Let $\FB$ be defined as in~\eqref{symknapineq}, with $P$ sufficiently large (that depends on $n$), $L\in\left\{0,1,\dots, \left\lceil\frac{n}{2}\right\rceil-1\right\}$  and $b:=L+1/P$.
Then, the $d$-$\sos(\FB)$-relaxation requires $d\geq n-L$ for enforcing $\sum_{i=1}^n x_i\geq \lceil b \rceil$.
\end{theorem}



We remark that Grigoriev, Hirsch, and Pasechnik gave in~\cite{GrigorievHP_MMJ02} a very interesting and influential result that is related to our Theorem~\ref{th:maxsymconstr}, but \emph{significatively different} in terms of both, lower bounds and techniques. We defer the interested reader to Section~\ref{sect:grigoriev} for a discussion on this point, and for a more precise meaning of ``significatively different''.

The result in Theorem~\ref{th:maxsymconstr} is disappointing for at least two reasons: the considered CG-cut looks pathetically trivial, and the proof that $d$-$\sos(\FB)$ fails for small $d$ is relatively complicated (see \vaia{sect:toyd}).

On the other side, it would be sufficient to have in the ``bag'' $\spns$ the set of symmetric polynomials, i.e. polynomials which do not change under permutations of the variables, to promptly enforce this CG-cut within $\Ss$-$\sos(\FB)$.
%
The proof is basically the same as the one given in \eqref{eq:interpol}: 
{
\begin{align*}
&\sum_{i=1}^n x_i- \lceil b \rceil=
\overbrace{\left(\sum_{i=0}^n \sum_{I\subseteq [n]:|I|=i} \delta_I\right)}^{=1} \left(\sum_{i=1}^n x_i- \lceil b \rceil\right)
\overset{\text{\eqref{Kr5}}}{\equiv}
\sum_{i=0}^n  \overbrace{\left(\sum_{I\subseteq [n]:|I|=i} \delta_I\right)}^{\text{symmetric}}(i- \lceil b \rceil)\\
&\overset{\text{\eqref{Kr4}}}{\equiv}
\underbrace{\sum_{i=\lceil b \rceil}^n
\left(
\overbrace{
\sum_{I\subseteq [n]:|I|=i} \delta_I
}^{\text{symmetric}}\sqrt{i- \lceil b \rceil}\right)^2}_{s_0(x)} +
\underbrace{
\sum_{i=0}^{\lceil b \rceil-1}
\left(
\overbrace{
\sum_{I\subseteq [n]:|I|=i} \delta_I
}^{\text{symmetric}}
\sqrt{\frac{i- \lceil b \rceil }{i-b}}
\right)^2}_{s_1(x)}
\underbrace{\left(\sum_{i=1}^n x_i- b \right)}_{g_1(x)}.
\end{align*}
}
Note that $s_0(x)$ and $s_1(x)$ are sum of squares of \emph{symmetric} polynomials.
It is a well-known fact that every symmetric polynomial can be written uniquely as a polynomial
in the $n+1$ elementary symmetric polynomials (see e.g. \cite{Sturmfels:2008}). Therefore, it is sufficient to define $\Ss$ as the set of elementary symmetric polynomials to guarantee that $\sum_{i=1}^n x_i- \lceil b \rceil$ admits a $\Ss$-$\sos$ certificate.
We refer to \cite{FawziSP15,Gatermann200495,raymond2016symmetric} for other more interesting symmetric situations.
%

We emphasize that in this paper we show how to handle some \emph{asymmetric} situations by exploiting the problem structure, which is our main result.
\subsection{On a Related Result by Grigoriev, Hirsch, and Pasechnik}\label{sect:grigoriev}

Grigoriev, Hirsch, and Pasechnik (see Theorem 8.1 in~\cite{GrigorievHP_MMJ02}) gave a result related to Theorem~\ref{th:maxsymconstr}, but also significatively different as explained in this section. 
In~\cite{GrigorievHP_MMJ02}, the \emph{symmetric knapsack} is defined as follows:
\begin{align}\label{symknap}
\FB'=\{x\in \R^n: x_i^2-x_i=0  \ \forall i\in[n], \sum_{i=1}^n x_i= b\}.
\end{align}
Note that $\FB'$ is a more constrained version of the set $\FB$ defined in~\eqref{symknapineq}.

The Positivstellensatz Calculus~\cite{GrigorievHP_MMJ02} is a proof system for languages consisting of unsolvable systems of polynomial equations. Note that \eqref{symknap} is unsolvable when $b$ is a non-integral value. A degree $d$ infeasibility certificate consists of a set of degree $d$ polynomials, say $\{h_1\ldots,h_l\}$, and a derivation of $\sum_j h_j^2=-1$ from $\FB'$.
Let $\delta$ denote the step function defined as follows:
\begin{align*}
  \delta(x) =
  \begin{cases}
  2, & \mbox{if } x\not\in [0,n]; \\
  2k + 4, & \mbox{if } x\in [k, k + 1]\cup [n - k - 1, n - k] \mbox{ for all integers }0 \leq k < n/2.
  \end{cases}
\end{align*}
In~\cite{GrigorievHP_MMJ02} the following result is proved.
\begin{theorem}\cite{GrigorievHP_MMJ02}\label{th:grig} Any Positivstellensatz calculus refutation of the symmetric knapsack problem $\FB'$ (see~\eqref{symknap}) has degree $\min\{\delta(b), \left\lceil(n - 1)/2\right\rceil + 1\}$.
\end{theorem}

Notice that any Positivstellensatz Calculus lower bound for the more constrained set $\FB'$ gives a $\sos$ lower bound for the set $\FB$ defined in~\eqref{symknapineq}.
However, for $b< n/2$, the bounds given by Theorem~\ref{th:grig} (see~\cite{GrigorievHP_MMJ02}), when applied to set $\FB$, are weaker, and also considerably weaker than the ones provided by our Theorem~\ref{th:maxsymconstr}. For example, for any given constant $k$ and $b\in (k,k+1)$, the degree lower bound in Theorem~\ref{th:grig} is $2k+4=O(1)$, whereas by Theorem~\ref{th:maxsymconstr} the degree lower bound is $n-k$.

Regarding the technique, Theorem~\ref{th:maxsymconstr} is proved by building on a result
given in~\cite{KLMipco16}. The latter has been shown to be very powerful in several other situations (see~\cite{KLMipco16,KLMicalp16} for more examples).


Finally, we observe that the study of the number of levels necessary to strengthen \emph{inequalities},
as in Theorem~\ref{th:maxsymconstr}, is useful for analyzing the $\sos$ ability to strengthen convex combinations of valid covering inequalities, as explained at the beginning of Section~\ref{sect:simple_example}. Analyzing \emph{equalities}, like in \eqref{symknap}, is less appropriate for these purposes.

\section{$\sos$ Derivation of Pitch Inequalities for set cover}\label{sect:sossetcov}
In this section we consider set cover problems.
For a given $m\times n$ matrix $A$ with 0/1 entries, the feasible region $\F_A$ for the \textsc{Set Cover} problem is defined by:
\begin{align}\label{setcov}
\F_A=\{x\in \R^n: x_i^2-x_i=0\ \forall i\in[n], Ax\geq e\},
\end{align}
where $e$ is the vector of 1s.
We focus on the concept of \emph{pitch} introduced in~\cite{BienstockZ04,zuckerberg2004set}.
\begin{definition}\label{def:pitch}
  For any given inequality $a^\top x- a_0\geq 0$, with indices ordered so that $0<a_1\leq a_2\leq \dots \leq a_h$ and $a_j=0$ for $j>h$, its \emph{pitch} $\pi(a,a_0)$ is the minimum integer such that $\sum_{i=1}^{\pi(a,a_0)} a_i - a_0 \geq 0$.
\end{definition}

We start emphasizing that valid inequalities for $\F_A$ of pitch at most $\pi$ are ``hard'' to enforce within ``standard'' hierarchies of relaxations, and this happens already with the first non-trivial pitch value, namely $\pi=2$ as shown by the following example.
\begin{example}\label{ex:fc1}
  Consider a set cover instance defined by a full-circulant constraint matrix $\FC$ as follows:
\begin{align}\label{eq:fullcirculant}
\F_{\FC}=\{x\in \R^n: x_i^2-x_i=0 \ \forall i\in[n], \sum_{j\in [n]\setminus\{i\}}x_j\geq 1 \ \forall i\in[n]\}.
\end{align}
%
Observe that $\sum_{j=1}^n x_j\geq 2$ is a pitch 2 valid inequality for the feasible region of this set cover instance. However, to enforce this inequality we need $n-3$ levels for a lifting operator stronger than the Sherali-Adams hierarchy \cite{BienstockZ04}, and requires at least $d=\Omega(\log^{1-\eps}n)$~\cite{KLMipco16}, with $\eps>0$ arbitrarily small, for the standard $d$-$\sos$ hierarchy (conjectured to be $n/4$ in \cite{BienstockZ04}).

This instance will be used in the following to exemplify our approach (see examples~\ref{ex:proof} and \ref{ex:lp}).
\end{example}

Vice versa, we show that there is a $\sap$-$\sos$ relaxation, where $\sap$ is a set of high degree polynomials of polynomial size, that satisfies all valid inequalities of constant pitch $\pi=O(1)$.
\begin{theorem}\label{th:setcover}
Consider a set cover problem given by a matrix $A$, and let $\pi=O(1)$ be a fixed non-negative integer. There is a polynomial-size $\sap$-$\sos$ relaxation that satisfies all valid inequalities for $\F_A$ of pitch at most $\pi$.
\end{theorem}
Note that the $\sap$-$\sos$ relaxation of Theorem~\ref{th:setcover} is completely determined by defining the set $\sap$ (see Section~\ref{sect:dual_lasserre} for a discussion on the size and on set of variables that appear in a generic $\sap$-$\sos$ relaxation). A closer look will reveal (see Section~\ref{sect:BZ-LP}) that the $\sap$-$\sos$ relaxation is actually a linear program corresponding to the generalized Sherali-Adams relaxation $\sap$-SA (see Section~\ref{sect:dual_lasserre}).

\paragraph{Preliminaries}
Given a vector $a\in \R^n$, the support of $a$, denoted $supp(a)$, is the set $\{i\in [n]: a_i\not=0\}$.
 Let $A_i\subseteq \{1,\ldots,n\}$ be the support of the $i$-th row of $A$. By overloading notation, we also use $A_i$ to denote the corresponding set of variables $\{x_j:j\in A_i\}$. We assume that $A$ is \emph{minimal}, i.e. there is no $i\not=j$ such that $A_i\subseteq A_j$.

For any given $T,F\subseteq [n]$ with $T\cap F=\emptyset$, let $\F_{A_{(T,F)}}$ denote the subregion of $\F_A$ with $x_i=1$ for all $i\in T$, and $x_j=0$ for all $j\in F$. Let $A_{(T,F)}$ be the matrix that is obtained from $A$ by removing all the rows where $x_i$ appears for $i\in T$ (these constraints are satisfied when $x_i=1$ for all $i\in T$) and setting to zero the $j$-th column for all $j\in F$. We will assume that $A_{(T,F)}$ is minimal by removing the dominated rows. Therefore, $\F_{A_{(T,F)}}=\{x\in \{0,1\}^n:A_{(T,F)}x\geq e, x_i=1  \ \forall i\in T, x_j=0  \ \forall j\in F\}$ and $\F_{A_{(T,F)}}\subseteq \F_{A}$.

For the sake of simplicity, we add the non-negative constraints $x_i\geq 0$ for $i\in[n]$ to the set of valid constraints that define the semialgebraic set \eqref{setcov}. This is not strictly necessary, since $x_i=x_i^2$ and therefore $x_i\geq 0$, but it will simplify the exposition.

 \subsection{Proof of Theorem~\ref{th:setcover}}
Let $a^\top x-a_0\geq 0$ be a valid inequality over $\F_A$ of pitch $\pi(a,a_0)\leq \pi$, with $a\geq 0$. First, we show a $\sos$ certificate of non-negativity for $a^\top x-a_0$. Then, we collect the polynomials we used in the $\sos$-certificate and put them in the ``bag'' $\sap$. So, the set of polynomials $\sap$ of Theorem~\ref{th:setcover} will be completely defined at the end of this proof, and its definition will naturally follow from the given $\sos$ certificate.

For the time being, it is sufficient to say that $\sap$ is a set of polynomials of size $(mn)^{O(1)}$, for any fixed $\pi=O(1)$. In $\sap$ every polynomial has the following form: $\sum_{J\in W}\delta_J^V$ for some $V\subseteq [n]$ and $W\subseteq 2^V$. In short, we will say that set $\sap$ is \emph{delta-structured} to denote this structure.

 By \eqref{Kr4}, note that $\sum_{J\in W}\delta_J^V\equiv (\sum_{J\in W}\delta_J^V)^2$, and therefore $q(x)\equiv q(x)^2$ for all $q(x)\in \sap$. Moreover, every polynomial in $\sap$ is non-negative over the Boolean hypercube. In the remainder a certificate of non-negativity will be congruent (mod $\IB$) to the following form:
\begin{align}\label{eq:certificate}
  \sum_{i} q_i(x) \overbrace{\left(\lambda_i^\top (Ax - e) + \gamma_i^\top x+\mu_i\right)}^{\text{conical combination of constraints}}, \qquad \text{ for some } q_i(x)\in \sap, \lambda_i,\gamma_i,\mu_i\geq 0 .
\end{align}
By the above properties, this certificate can be immediately transformed into a $\sap$-$\sos$ certificate.

The proof of Theorem~\ref{th:setcover} will be by induction on the pitch value $\pi$. The base of the induction $\pi=0$ is trivial: in this case we must have  $a_0\leq 0$, and $\Ss_A(0)=\{1\}$ is sufficient to prove that $-a_0\geq 0$. Note that $\Ss_A(0)=\{1\}$ is independent on the matrix $A$ and it is delta-structured (recall if $V=\emptyset$ then $\delta_J^V=1$).

By induction hypothesis, for any given $0\leq p\leq \pi-1$ and any given constraint matrix $A'$,  we assume that any valid pitch-$p$ inequality for $\F_{A'}$ admits a $\Ss_{A'}(p)$-$\sos$ certificate where $\Ss_{A'}$ is delta-structured. We will prove that the induction hypothesis also holds for pitch $\pi$ (induction step).

We proceed ``backwards'', as in ~\eqref{eq:interpol}. We start multiplying $a^\top x-a_0$ by $\sum_{I\subseteq V}\delta_I^V$, for a suitably chosen set $V\subseteq [n]$ that will be specified soon. Recall that $\sum_{I\subseteq V}\delta_I^V=1$ (see \eqref{Kr1}). Let $(a^\top x-a_0)_{(T,F)}$ denote $(a^\top x-a_0)$ after setting $x_i=1$ for $i\in T$ and $x_j=0$ for $j\in F$. By~\eqref{Kr5}, note that $\delta_J^V  (a^\top x-a_0)\equiv \delta_J^V  (a^\top x-a_0)_{(J,V\setminus J)}$.
Let $\delta_{\geq\pi}^V:=\sum_{I\subseteq V, |I|\geq \pi}\delta_{I}^V$ (zero if $|V|<\pi$).
 It follows that:
{
\begin{align}
&a^\top x-a_0 = \overbrace{\left(\sum_{I\subseteq V} \delta_{I}^V \right)}^{=1}(a^\top x-a_0)\nonumber\\
&\equiv \underbrace{\delta_{\emptyset}^V (a^\top x-a_0)_{(\emptyset,V)}}_{\textsc{First}}+\underbrace{ \left(\sum_{J\subseteq V,0<|J|< \pi} \delta_{J}^V (a^\top x-a_0)_{(J,V\setminus J)}\right)}_{\textsc{Second}} + \underbrace{(\delta_{\geq\pi}^V)(a^\top x-a_0)}_{\textsc{Third}}\label{eq:split}.
\end{align}
}
Therefore, showing a $\sos$ certificate for $a^\top x-a_0$ boils down to provide a $\sos$ certificate for each of the three summands, \textsc{First}, \textsc{Second} and \textsc{Third}, in~\eqref{eq:split}. Before doing this we need to specify the set $V\subseteq [n]$.
\paragraph{How to Choose $V$}
Set $V$ is chosen according to the following Lemma~\ref{th:bzcore} (see \cite{BienstockZ04,zuckerberg2004set}) which gives a structural property of valid inequalities for set cover. The statement of Lemma~\ref{th:bzcore} is slightly different from Proposition 4.22 in~\cite{zuckerberg2004set} (or Theorem 6.3 in~\cite{BienstockZ04}). The main difference is given by Property~\eqref{eq:bz4} (see Lemma~\ref{th:bzcore}). This property is not explicitly given in~\cite{BienstockZ04,zuckerberg2004set}, but it can be easily derived by their construction as explained in the proof that follows.
%
%
\begin{lemma}\cite{BienstockZ04,zuckerberg2004set}\label{th:bzcore}
Suppose $a^\top x-a_0\geq 0$ is a valid inequality for $\F_A$ with $a\geq 0$. Then there is a subset $C=C(a,a_0)$ of the rows of $A$ with $|C|\leq \pi(a,a_0)$, such that
\begin{align}
&A_i\subseteq \supp(a), \quad \forall i\in C, \label{eq:bz1}\\
&(a^\top x- a_0)_{(\emptyset,V)}\geq 0 \text{ is valid for }\F_C,\label{eq:bz3}\\
&\F_{A_{(\emptyset,V)}}\not=\emptyset, \label{eq:bz4}
\end{align}
where $V:= \bigcup_{{i,j\in C, i\not=j}} A_i\cap A_j$ is the set of variables occurring in more than one row of~$C$, and $\F_C:=\{x\in [0,1]^n:(\sum_{j\in A_i}x_j-1)_{(\emptyset,V)}\geq 0  \ \forall i\in C\}$.
\end{lemma}
%
\begin{proof}

The proof is by induction on $\pi=\pi(a,a_0)$. If $\pi=0$ then $|C|=0$, and it follows that $\F_C=\{x:x\in [0,1]^n\}$ and $V=\emptyset$. A pitch zero inequality must have $a_0\leq 0$. So, since $a\geq 0$, $a^\top x- a_0\geq 0$ is indeed valid for $\F_C$ and for $\F_{A_{(\emptyset,\emptyset)}}=\F_{A}$ ($\not=\emptyset$).

Now, assume that the claim holds for all valid inequalities of pitch $p$ with $0\leq p\leq \pi -1$ and $\pi\geq 1$. Consider a valid inequality $a^\top x-a_0\geq 0$ of pitch $\pi$. Note that there must be some $v\in[m]$ such that $A_v\subseteq \supp(a)$  or, otherwise, we could set $x_j=0$ for all $j\in \supp(a)$, and $x_j=1$ everywhere else, and thereby satisfy every constraint and nevertheless have $a^\top x =0$ (so contradicting the hypothesis that $a^\top x-a_0\geq 0$ is a valid inequality of pitch $\pi\geq 1$). Choose $A_v\subseteq \supp(a)$. Note that we are assuming, w.l.o.g., that $A$ is minimal, so there is no $A_i$, with $i\in[m]$ and $i\not=v$, that is a proper subset of $A_v$. Let $v(1)\in A_v$ be the index of the minimum coefficient $a_j:j\in A_v$, where $a_j$ is the coefficient of variable $x_j$ in the valid inequality $a^\top x-a_0\geq 0$.

We first obtain a strengthen by setting to zero all the variables from $V_v$, where $V_v$ are all the variables from all $A_i$, with $i\not=v$, that appear in $A_v-\{v(1)\}$, i.e. $V_v:= (A_v-\{v(1)\})\bigcap(\cup_{i\not=v} A_i)$. Consider $\F_{A_{(\emptyset,V_v)}}$ and note that $\F_{A_{(\emptyset,V_v)}}\not=\emptyset$ because by assumption no $A_j\subset A_v$ and therefore $(a^\top x- a_0)_{(\emptyset,V_v)}\geq 0$ is a valid inequality for $\F_{A_{(\emptyset,V_v)}}$.
 Set $x_{v(1)}=1$ in $(a^\top x- a_0)_{(\emptyset,V_v)}\geq 0$ to get $(a^\top x- a_0)_{(\{v(1)\},V_v)}\geq 0$ which is a valid inequality for $\F_{A_{(\emptyset,V_v)}}$. Note that the pitch $p$ of $(a^\top x- a_0)_{(\{v(1)\},V_v)}$ is such that $p\leq \pi-1$ and therefore, by induction hypothesis, it satisfies the properties of the claim when we consider $(a^\top x- a_0)_{(\{v(1)\},V_v)}\geq 0$ as valid inequality for $\F_{A_{(\emptyset,V_v)}}$.
 Let $a'$ be the vector that is obtained from $a$ by setting to zero all the coefficients from $V_v\cup\{v\}$ and let $a_0':=a_0-a_{v(1)}$, so $(a^\top x- a_0)_{(\{v(1)\},V_v)}=a'^\top x-a_0'$. By the induction hypothesis there must be a subset $C'$ of the rows from  $A':=A_{(\emptyset,V_v)}$ such that $|C'|\leq p$ and
\begin{align}
&A'_i\subseteq \supp(a'), \quad \forall i\in C',\\
&(a'^\top x- a'_0)_{(\emptyset,V')}\geq 0 \text{ is valid for }\F_{C'},\\
&\F_{A'_{(\emptyset,V')}}\not= \emptyset,\label{eq:bz4'}
\end{align}
where $V'$ is the set of variables occurring in more than one row from $C'$ and $\F_{C'}=\{x\in [0,1]^n:(\sum_{j\in A'_i}x_j-1)_{(\emptyset,V')}\geq 0, i\in C'\}$.

Define $C:=\{v\}\cup C'$. Therefore Condition~\eqref{eq:bz1} is satisfied by construction.
Moreover, note that in $\F_C$ (as defined in the statement of Lemma~\ref{th:bzcore}) all the constraints are disjoint, and basic feasible solutions are integral (in case needed, we refer to~\cite{zuckerberg2004set} for more details).
%
Suppose that we are given an arbitrary $\tilde{x}\in\{0,1\}^n$ that satisfies $\F_C$. Consider that we must have $\tilde{x}_j=1$ for some $j\in A_v$, such that $a_j\geq a_{v(1)}$. If we define $x'$ to be the same as $\tilde{x}$ but with $x'_j=0$, then $x'$ still satisfies $\F_{C'}$. Thus by induction $a'^\top x'\geq a_0-a_{v(1)}$ which implies that $a^\top \tilde{x}= a'^\top x'+a_j \tilde{x}_j=a'^\top x'+a_j \geq a_0-a_{v(1)}+a_j\geq a_0$. This proves Property~\eqref{eq:bz3}.

To prove Property~\eqref{eq:bz4} we show that we can set to zero all the overlapping variables from the rows in $C$, namely the variables from $V$ and still get a non empty set of integral solutions, i.e. $\F_{A_{(\emptyset,V)}}\not=\emptyset$.
Indeed, by the induction hypothesis we have that $\F_{A'_{(\emptyset,V')}}\not=\emptyset$, where $A'_{(\emptyset,V')}=A_{(\emptyset,V_v\cup V')}$. Therefore $\F_{A_{(\emptyset,V)}}\not=\emptyset$ because $V\subseteq V_v\cup V'$.
\end{proof}

\paragraph{\textsc{First} $\sos$ Certificate}
Consider the \textsc{First} summand in~\eqref{eq:split}. 
By Lemma~\ref{th:bzcore}, we have that $(a^\top x- a_0)_{(\emptyset,V)}\geq 0$ is valid for $\F_C$ (see \eqref{eq:bz3}). Note that the linear constraints that define the feasible region $\F_C$ are just a subset of the linear constraints from $\F_A$ after setting to zero all the variables from $V$. It follows that
 $(a^\top x- a_0)_{(\emptyset,V)}=\left(\lambda^\top (Ax - e) + \gamma^\top x+\mu\right)_{(\emptyset,V)}$
for some $\lambda,\gamma,\mu \geq 0$.
Then:
{\small
\begin{align*}
\delta_{\emptyset}^V (a^\top x-a_0)_{(\emptyset,V)}
&= \delta_{\emptyset}^V \left(\lambda^\top (Ax - e) + \gamma^\top x+\mu\right)_{(\emptyset,V)}
\overset{\text{\eqref{Kr5}}}{\equiv} \delta_{\emptyset}^V \left(\lambda^\top (Ax - e) + \gamma^\top x+\mu\right).
\end{align*}
}
The latter has the form given by~\eqref{eq:certificate}, and it yields a $\sos$ certificate.
In order to obtain such a certificate it is sufficient to include in $\sap$ the multilinear polynomial $\delta_\emptyset^V$. With this aim, by using Lemma~\ref{th:bzcore}:
let $\C(\pi):=\{C:C\subseteq [m]\wedge |C|\leq \pi\}$ and $V_C:= \bigcup_{\substack{i,j\in C, i\not=j}} A_i\cap A_j$ be the set of variables occurring in more than one row with index from $C\in \C(\pi)$; Add to $\sap$ all $\delta_\emptyset^{V_C}$ with $C\in \C(\pi)$.  For any given constant pitch $\pi$, there are polynomially many such $\delta_\emptyset^{V_C}$, and one of them is equal to $\delta_\emptyset^V$ by Lemma~\ref{th:bzcore}.

\paragraph{\textsc{Second} $\sos$ Certificate}
Consider the \textsc{Second} summand in~\eqref{eq:split}.
By Property~\eqref{eq:bz4} we know that by setting to zero all the variables from $V$ we obtain a non-empty subset of feasible integral solutions. It follows that by setting $x_j=1$, for all $j\in J$, and $x_h=0$, for all $h\in V\setminus J$, we obtain a non-empty subset of feasible integral solutions, i.e. $\F_{A_{(J,V\setminus J)}}\not= \emptyset$ and $(a^\top x-a_0)_{(J,V\setminus J)}\geq 0$ is a valid inequality for the solutions in $\F_{A_{(J,V\setminus J)}}$ (since  $a^\top x-a_0\geq 0$ is by assumption a valid inequality for any feasible integral solution). Moreover the pitch $p$ of $(a^\top x-a_0)_{(J,V\setminus J)}\geq 0$ is strictly smaller than~$\pi$, $0\leq p\leq \pi-|J|$. By the induction hypothesis, it follows that $(a^\top x-a_0)_{(J,V\setminus J)}$ has a $\Ss_{A_{(J,V\setminus J)}} (p)$-$\sos$ certificate, which means that there is a $q(x)\in \Ss_{A_{(J,V\setminus J)}}(p)$ such that
\begin{align*}
  (a^\top x-a_0)_{(J,V\setminus J)}\equiv q(x) \left(\lambda_J^\top (Ax - e) + \gamma_J^\top x+\mu_J\right)_{(J,V\setminus J)},
\end{align*}
for some $\lambda_J,\gamma_J,\mu_J \geq 0$.
The claim follows by observing that
$$\delta_{J}^V q(x) \left(\lambda_J^\top (Ax - e) + \gamma_J^\top x+\mu_J\right)_{(J,V\setminus J)}\overset{\text{\eqref{Kr5}}}{\equiv} \delta_{J}^V q(x) \left(\lambda_J^\top (Ax - e) + \gamma_J^\top x+\mu_J\right).$$
Again, the latter has the form given by~\eqref{eq:certificate}.
Note that $\delta_{J}^V q(x)$ is delta-structured. We define the set $\sap$ so that it includes $p(x):=\delta_J^{V_C}\cdot q(x)$ for all $q(x)\in \Ss_{A_{(J,V_C\setminus J)}} (\pi-|J|)$ and for all $J\subseteq V_C,0<|J|< \pi$ and $C\in \C(\pi)$.

\paragraph{\textsc{Third} $\sos$ Certificate}

Finally, consider the \textsc{Third} summand from~\eqref{eq:split}.
Recall, see Definition~\ref{def:pitch}, that $0<a_1\leq a_2\leq \cdots \leq a_h$ and $a_j=0$ for $j>h$ for some $h\in [n]$, so $\supp(a)=\{1,\ldots,h\}$. 
By \eqref{eq:bz1}, $V\subseteq \supp(a)$. If $|V|<\pi$ then $\delta_{\geq \pi}^V$ is the null polynomial and we are done. Otherwise, let $a'_i := a_i$ for $i\in[\pi]$, $a_i':=a_\pi$ for $i=[h]\setminus[\pi]$ and $a_i':=0$ for $i\in \supp(a)\setminus V$.
%
It follows that
\begin{align*}
&\delta_{\geq \pi}^V\left(\sum_{i=1}^h a_i x_i-a_0\right)=\delta_{\geq \pi}^V\left(\sum_{i\in V} a_i x_i-a_0 + \sum_{i\in \supp(a)\setminus V} a_i x_i\right)\\
&=\delta_{\geq \pi}^V\left(\sum_{i\in V} a_i' x_i-a_0 + \sum_{i\in \supp(a)} (a_i-a_i') x_i\right)\\
%
%
&\overset{\text{\eqref{Kr5}}}{\equiv} \sum_{I\subseteq V\cap [\pi]} \ \sum_{k=\pi-|I|}^{|V|} \ \left(\overbrace{\sum_{\substack{J\subseteq V\setminus [\pi]\\ |J|=k}}\delta_{I\cup J}^V}^{p_{I,k}(x)} \left(\overbrace{\sum_{i\in I} a_i' +k a_\pi -a_0}^{\geq 0}+\sum_{i\in \supp(a)} (a_i-a_i') x_i\right)\right).
%
%
\end{align*}
The latter has again the form given by~\eqref{eq:certificate}, and it yields a $\sos$ certificate.  We define the set $\sap$ so that it includes the polynomials $p_{I,k}(x)$ that are used in the above formula. Note that each $p_{I,k}(x)$ is a symmetric polynomial with respect to the variables indexed by set $V\setminus [\pi]$; therefore it admits a succinct representation by the mean of elementary symmetric polynomials.

\subsubsection{Set $\sap$} \label{sect:sap}
We summarize the definition of $\sap$.
Let
\begin{align}
  \C(\pi)&:=\{C:C\subseteq [m]\wedge |C|\leq \pi\},\label{eq:cp} \\
  V_C&:= \bigcup_{\substack{i,j\in C, i\not=j}} A_i\cap A_j. \label{eq:vc}
\end{align}
Set $\sap$ includes the following polynomials:
{\small
\begin{align*}
  &\left\{\delta_\emptyset^{V_C} : C\in \C(\pi) \right\},  &\textsc{(First)}\\
  &\left\{ \delta_J^{V_C}\cdot q(x) : C\in \C(\pi), J\subseteq V_C \text{ with }0<|J|< \pi, q(x)\in \Ss_{A_{(J,V_C\setminus J)}} (\pi-|J|)\right\},  &\textsc{(Second)}\\
  &\left\{\sum_{\substack{J\subseteq V_C\setminus [\pi]\\ |J|=k}}\delta_{I\cup J}^V  : C\in \C(\pi),I\subseteq V_C \text{ with } |I|\leq \pi,k=\pi-|I|,\ldots,|V_C|\right\}. &\textsc{(Third)}
\end{align*}
}
Note that when $\pi\in \{0,1\}$ then $\sap=\{1\}$.
By a simple counting argument, we have $|\sap|=(mn)^{O(1)}$ for any fixed $\pi=O(1)$.

\begin{example}[pitch 2 certificate]\label{ex:proof}
    Consider the set cover instance defined by \eqref{eq:fullcirculant}, in Example~\ref{ex:fc1}, namely by a full-circulant constraint matrix FC. The entries of the $i$-th row of matrix  $FC$ are all equal to 1 but the $i$-th entry that is zero.
  Let $\FC_i:=[n]\setminus\{i\}$ denote the support of the $i$-th row of matrix $\FC$. Let $g_i(x):=\sum_{j\in \FC_i}x_j- 1\geq 0$ be the $i$-the constraint corresponding to row $\FC_i$.

   As already observed, $\sum_{j\in[n]} x_j\geq 2$ is a pitch 2 valid inequality for the feasible region of this set cover instance, and this inequality is ``hard'' to enforce by ``standard'' hierarchies like Lasserre/$d$-$\sos$ and $d$-SA (Sherali-Adams).

  We start considering the spanning set $\Ss_{\FC}(2)$ (defined in Section~\ref{sect:sap}). According to the definition of set $\C(2)$, see \eqref{eq:cp}, note that $\{1,2\}\in \C(2)$; then, see \eqref{eq:vc}, $V_{\{1,2\}}=\FC_1\cap \FC_2=\{3,\ldots,n\}$. For short let $V:=V_{\{1,2\}}$. The following set $\mathcal{P}$ of polynomials is a subset of $\Ss_{\FC}(2)$:
  \begin{align}\label{p}
\mathcal{P}:= \overbrace{\left\{\delta_\emptyset^{V} \right\}}^{\mathcal{P}_0}\cup
    \overbrace{\left\{\delta_{\{i\}}^{V}: i\in V \right\}}^{\mathcal{P}_1} \cup
    \overbrace{\left\{\sum_{\substack{I\subseteq V: \\|I|=k }}\delta_{I}^{V}: k=2,\ldots,n \right\}}^{\mathcal{P}_2}\subseteq \Ss_{\FC}(2).
  \end{align}
In addition to those listed above, note that in set $\Ss_{\FC}(2)$ there are also other polynomials. These polynomials are all the same under a permutation of the variables and they play a similar role due to the symmetry of the  example.
By using the above polynomials we obtain a proof of non-negativity as follows.

\begin{align}
&\sum_{j\in[n]} x_j- 2 =
\overbrace{
\left(
 \delta_\emptyset^{V}
+ \sum_{i\in V} \delta_{\{i\}}^{V}
+ \sum_{k=2}^{n} \  \left(\sum_{\substack{I\subseteq V:\\ |I|=k }}\delta_{I}^{V}\right)
 \right)
 }^{=1}\left(\sum_{j\in[n]} x_j- 2\right)\nonumber
 \\
%
%
%
%
&\equiv \underbrace{\delta_{\emptyset}^V
\left(x_1+x_2- 2\right)}_{\textsc{First}}
+
\underbrace{ \left(\sum_{i\in V} \delta_{\{i\}}^V \left(x_1+x_2 - 1\right)\right)}_{\textsc{Second}}
+
\underbrace{\sum_{k=2}^{n}\left( \sum_{\substack{I\subseteq V:\\ |I|=k }}\delta_I^V \right)\left(x_1+x_2+k- 2\right)}_{\textsc{Third}}\\
%
%
&\equiv \underbrace{\delta_{\emptyset}^V
\left(g_1(x)+g_2(x)\right)}_{\textsc{First}}
+
\underbrace{ \left(\sum_{i\in V} \delta_{\{i\}}^V g_i(x)\right)}_{\textsc{Second}}
+
\underbrace{\sum_{k=2}^{n}\left( \sum_{\substack{I\subseteq V:\\ |I|=k }}\delta_I^V \right)\left(x_1+x_2+(k- 2)\right)}_{\textsc{Third}}.\nonumber
\end{align}
The latter has the form given by~\eqref{eq:certificate}, and it yields a $\sos$ certificate (and it is a SA certificate) for the considered pitch 2 inequality.
\end{example}

\subsection{An Explicit Compact LP Formulation}\label{sect:BZ-LP}
For any fixed $\pi=O(1)$, in the proof of Theorem~\ref{th:setcover} we have shown that every valid inequality $a^\top x-a_0\geq 0$ of pitch at most $\pi$ admits a certificate of non-negativity that is congruent (mod $\IB$) to~\eqref{eq:certificate}. By reformulating this result in an equivalent way, we have shown that $a^\top x-a_0$ belongs to the following cone of polynomials:
{
\small
\begin{align}\label{eq:sccone}
 \C_{\Ss_A(\pi)}=\left\{ \overline{\sum_{i} q_i(x) \left(\lambda_i^\top (Ax - e) + \gamma_i^\top x+\mu_i\right)}: q_i(x)\in \sap, \lambda_i,\gamma_i,\mu_i\geq 0\right\} .
\end{align}
}
The dual cone $\C_{\Ss_A(\pi)}^*$ is the set of linear functionals $\yy{\cdot}$ that are non-negative on the primal cone satisfying (see the discussion in Section~\ref{sect:dual_lasserre}, constraints \eqref{eq:l1}, \eqref{eq:SAcond1}, \eqref{eq:SAcond2}):
%
\begin{align}
& \yy{1}=1;\label{eq:sc1}\\
& \yy{\overline{q(x)}} \geq 0, \quad \forall q(x) \in \Ss_A(\pi); \label{eq:sc2}\\
& \yy{\overline{q(x)\cdot  g_i(x)}} \geq 0, \quad \forall q(x) \in \Ss_A(\pi), \forall i\in[m+n];  \label{eq:sc3}
 \end{align}
where $g_i(x)\geq 0$, for $i\in[m+n]$, denotes a constraint from $Ax\geq e$, or $x_j\geq 0$ for $j\in [n]$.

As already discussed in Section~\ref{sect:dual_lasserre}, the linear functional inequalities \eqref{eq:sc1}, \eqref{eq:sc2} and \eqref{eq:sc3}, yield a \emph{linear program} of size $O(|\Ss_A(\pi)|)$. It is actually a \emph{hierarchy} of linear programs parameterized by the pitch~$\pi$. This relaxation can be seen as a generalized Sherali-Adams relaxation, where the standard monomial basis of degree $\leq d$ has been replaced with the set $\Ss_A(\pi)$ of high degree polynomials.

\begin{example}[Pitch 2 LP]\label{ex:lp}
We provide an explicit LP for the set cover instance considered in examples~\ref{ex:fc1} and \ref{ex:proof}.
With this aim, we can either compute an ordered basis for the cone of polynomials \eqref{eq:sccone}, or alternatively, an ordered spanning set and impose the linearity conditions (see the discussion in Section~\ref{sect:dual_lasserre} and Condition~\eqref{eq:linearity}). Here, we follow the second option.

Let $T:=\{\overline{x_i\cdot p}: p\in \Ss_{\FC}(2), i\in[n]\cup\{0\}\}$, and note that $T$ is a spanning set for~\eqref{eq:sccone}. The dimension of $T$ is equal to the dimension of the linear functionals $y$: there is one entry in $y$ for each polynomial in $T$. So vector $y$ is indexed by the polynomials in~$T$.
Consider set $\mathcal{P}\subseteq \Ss_{\FC}(2)$ of polynomials (see \eqref{p}).
\begin{itemize}
  \item \emph{Variables}. The LP variables are the entries of vector $y$. In particular there are the following variables: $\yy{\overline{q(x) x_j}}$, for $q(x)\in \Pp=\Pp_0\cup \Pp_1 \cup \Pp_2$ and $j\in[n]\cup \{0\}$ (recall $x_0:=1$).
  \item \emph{Constraints}. By~\eqref{eq:sc2},\eqref{eq:sc3} we have the following linear constraints in the LP formulation:
      \begin{align}
& \yy{\overline{q(x)x_j}} \geq 0, \quad \forall q(x) \in \mathcal{P}_2, j=0,1,2; \label{eq:FC1}\\
& \yy{\overline{\delta_\emptyset^V \cdot  g_i(x)}} = \yy{\delta_\emptyset^V  x_i}-\yy{\delta_\emptyset^V} \geq 0, \quad \forall i=1,2;\label{eq:FC2}\\
& \yy{\overline{\delta_{\{i\}}^V \cdot  g_i(x)}} = \yy{\delta_{\{i\}}^V x_1}+\yy{\delta_{\{i\}}^V x_2}-\yy{\delta_{\{i\}}^V} \geq 0, \quad \forall i\in V=\{3,\ldots,n\}\label{eq:FC3}.
\end{align}

\end{itemize}
The following valid inequality can be obtained by a conical combination of \eqref{eq:FC1}-\eqref{eq:FC3}:
\begin{align}\label{eq:temp}
\begin{split}
\yy{\delta_{\emptyset}^V x_1}&+ \yy{\delta_{\emptyset}^V x_2}-2 \yy{\delta_{\emptyset}^V}
+
\sum_{i\in V}\left( \yy{\delta_{\{i\}}^V x_1}+ \yy{\delta_{\{i\}}^V x_2} - \yy{\delta_{\{i\}}^V}\right)
+\\
&+\sum_{q(x)\in \Pp_2}\left(\yy{q(x) x_1}+\yy{q(x) x_2}+(k-2)\yy{q(x)}\right)\geq 0.
\end{split}
\end{align}
Note that $\sum_{q(x)\in \Pp}q(x)=1$, and therefore by the linearity conditions (see the discussion in Section~\ref{sect:dual_lasserre} and Condition~\eqref{eq:linearity})
the following is part of the set of the LP constraints (for $j=1,2$):
\begin{align*}
  \yy{\delta_{\emptyset}^V x_j}+ \sum_{i\in V} \yy{\delta_{\{i\}}^V x_j}
  +
  \sum_{q(x)\in \Pp_2}\yy{q(x) x_j} = y[x_j].
\end{align*}
Analogously,
note $\sum_{k=1}^{n}k\left(\sum_{\substack{I\subseteq V: \\|I|=k }}\delta_{I}^{V}\right)= \sum_{i\in V} x_i$, which by linearity, gives the following constraint that holds for the linear functional $y$ (and that is part of the LP formulation):
$$\sum_{k=1}^{n}k\cdot \yy{\sum_{\substack{I\subseteq V: \\|I|=k }}\delta_{I}^{V}}= \sum_{i\in V} \yy{x_i}.$$
Then, \eqref{eq:temp} and the linearity conditions imply the pitch 2 inequality $$\sum_{i\in [n]} \yy{x_i}-2\geq 0.$$
\end{example}


\section{The Bienstock-Zuckerberg Hierarchy}\label{sect:bz}
The Bienstock-Zuckerberg hierarchy (BZ)~\cite{BienstockZ04,zuckerberg2004set} generalizes the approach for set cover. The full description requires several layers of details and here we sketch only the main points. We refer to the original manuscripts for a more precise and comprehensive  description.

Any non-trivial constraint can be rewritten in the set cover form: $\sum_{i\in I} a_i x_i +\sum_{j\in J} a_j (1-x_j)\geq b$, with all the coefficients $a,b$ non-negative. Then the BZ hierarchy uses the standard concept of minimal covers\footnote{More precisely, in~\cite{BienstockZ04,zuckerberg2004set} a closely related concept that is called \emph{obstruction} is used.} (see e.g. \cite{Conforti:2014}): a \emph{minimal cover} is an inclusion-minimal set $C\subseteq \supp(a)$ such that $\sum_{j\not\in C} a_j<b $ and therefore $\sum_{j \in C} x_j' \geq 1$ is a valid inequality (where $x_j'=x_j$ if $j\in I$ or $x_j'=1-x_j$ else).
%
 In general, the number of minimal covers can be exponential so the idea in BZ is to generate only the ``$k$-small'' ones, which are added to the original relaxation. Here with ''$k$-small'' we mean all the valid minimal covers with all the variables from $I$ (or $J$) but at most $k$, or at most $k$ from $I$ (or $J$). These minimal covers can be enumerated in polynomial time for any fixed $k$.
 Then the set cover approach is applied to the set cover problem given by the $k$-small minimal covers. If the minimal covers are polynomially bounded this allows to generate the pitch bounded valid inequalities as for set cover (see the application below).
 Roughly speaking, the ``power'' of the BZ approach is given by the presence of the $k$-small minimal covers, if this set is empty then the hierarchy is not stronger than a variant of the Sherali-Adams hierarchy (see \cite{AuT18}).

 The BZ approach can be reframed into the $\sos$ framework by choosing the appropriate spanning polynomials. We omit the complete mapping because this would require the full description of BZ that is quite lengthy. Moreover, the most important application of BZ currently known is given by the set cover problem, which has been widely explained in previous sections.  By way of example, we show in \vaia{sect:smallmincovers} that we do not need to explicitly add the $k$-small minimal covers, since they can be implied by adding the ``right'' polynomials. By using the explained ideas, it should be easy to fill in the missing details.

\section{Chv\'{a}tal-Gomory Cuts}\label{sect:cg}

Consider a rational polyhedra $P=\{x\in \R^n: Ax\geq b\}$ with $A\in \Zz^{m\times n}$ and $b\in \Zz^{m}$. Inequalities of the form $(\lambda^\top A)x\geq \lceil \lambda^\top b\rceil$, with $\lambda\in \R_+^m$, $\lambda^\top A\in \Zz^n$, and $\lambda^\top b\not\in \Zz$, are commonly referred to Chv\'{a}tal-Gomory cuts (CG-cuts for short), see e.g. \cite{Conforti:2014}. CG-cuts are valid for the integer hull $P^*$ of $P$.

The following rational polyhedron is commonly referred to as the \emph{first CG closure}:
%
%
\begin{align}\label{cg-def}
P^{(1)}:=\{x\in \R^n:(\lambda^\top A)x\geq \lceil \lambda^\top b\rceil, \lambda\in [0,1]^m, \lambda^\top A\in \Zz^n\}.
\end{align}
In particular $P^{(1)}$ is a stronger relaxation of $P^*$ than $P$, i.e. $P^*\subseteq P^{(1)}\subseteq P$.
We can iterate the closure process to obtain the CG closure of $P^{(1)}$. We denote by $P^{(2)}$ this second CG closure. Iteratively, we define the $t$-th CG closure $P^{(t)}$ of $P$ to be the CG closure of $P^{(t-1)}$, for $t\geq 2$ integer. An inequality that is valid for $P^{(t)}$ but not for $P^{(t-1)}$ is said to have \emph{CG-rank} $t$.

Eisenbrand and Schulz \cite{EisenbrandS03} proved that for any polytope $P$ contained in the unit cube $[0, 1]^n$, one can choose $t = O(n^2 \log n)$ and obtain the integer hull $P^{(t)}=P^*$. Rothvo{\ss} and Sanit\'{a} \cite{RothvossS13} proved that there is a polytope contained in the unit cube whose CG-rank has order $n^2$, thus showing that the above bound is tight, up to a logarithmic factor.

The CG-cuts that are valid for $P^{(1)}$ and that can be derived by using coefficients in $\lambda$ of value $0$ or $1/2$ only, are called $\{0,1/2\}$-cuts. In~\cite{LetchfordPS11} it is shown that the separation problem for $\{0,1/2\}$-cuts  remains strongly NP-hard, even when all integer variables are binary, $P=\{x\in \R_+^n: Ax\leq e\}$ with $A\in \{0,1\}^{m\times n}$, and $e$ denote the all-one vector with $m$ entries. As pointed out in~\cite{LetchfordPS11}, the latter hardness proof can easily be adapted to set partitioning and set cover problems. This result implies that it is NP-hard to optimize a linear function over the first closure $P^{(1)}$. 

For min set cover problems, Bienstock and Zuckerberg~\cite{BienstockZ06} obtained the following result.
For an arbitrary fixed precision $\eps>0$ and a fixed $t\in \N$, choose $\pi$ such that $\left(\frac{\pi+1}{\pi}\right)^t\leq 1+\eps$. For any given set cover instance, let $opt$ denote the optimal integral value and let $opt^{(t)}$ ($\leq opt$) denote the optimal value over the $t$-th closure $P^{(t)}$. Bienstock and Zuckerberg~\cite{BienstockZ06} considered the optimal solution $x^*_\pi$ of value $opt_\pi$ ($\leq opt$) of a relaxation $R_\pi$ that  satisfies all pitch-$\pi$ valid inequalities for the integer hull.
Then, either ($opt\geq$) $opt_\pi \geq opt^{(t)}$, implying therefore that $R_\pi$ is a better relaxation than the $t$-th closure $P^{(t)}$, or ($opt \geq$) $opt^{(t)} \geq opt_\pi$. In the latter case, they proved that $x_\pi^*$ can be rounded to satisfy all the CG-cuts of rank $t$. Moreover, the value of the rounded solution is at most $1+\eps$ times larger than $opt_\pi$. This implies that $(1+\eps) opt_\pi\geq opt^{(t)}$ and therefore $opt_\pi\geq (1-\eps) opt^{(t)}$.
This gives a polynomial time approximation scheme (PTAS) for approximating $opt^{(t)}$, i.e. for the minimization of set cover objective functions over $P^{(t)}$.
It follows that the generalized $\sos$ (or Sherali-Adams) relaxation with high degree polynomials described in this paper, yields also to a PTAS for approximating set cover objective functions over $P^{(t)}$.


In the next section we present a more general result for packing problems, meaning that the coefficients of the non-negative matrix $A$ are not anymore restricted to be 0/1, or bounded (see \vaia{sect:smallmincovers}), as for the set cover case. It remains an interesting open question to extend the results for the set cover problem to general covering problems, namely covering problems with general non-negative matrices $A$.

\subsection{Approximating Fixed-Rank CG Closure for Packing Problems}\label{sect:cgpack}



In this section we consider packing problems and show that $d$-$\sos$
yields a PTAS for approximating over the $t$-th CG closure $P^{(t)}$, for any fixed $t$.
%
It follows that the $\sos$ approach can be used for approximating to any arbitrary precision, over any constant CG closure, for both packing and set cover problems (BZ guarantees this only for set cover problems). 

Consider any given $m\times n$ non-negative matrix $A$ and a vector $b\in \R_+^m$. Let $\F_{A,b}$ be the feasible region for the $0$-$1$ packing problem defined by $A$ and $b$:
$$\F_{A,b}=\{x\in \R^n: x_i^2-x_i=0  \ \forall i\in[n], Ax\leq b\}.$$
%

We extend the definition of pitch also for packing inequalities as follows.
\begin{definition}
For any given packing inequality $a_0- a^\top x\geq 0$, with $a_0,a\geq 0$ and indices ordered so that $0<a_1\leq a_2\leq \dots \leq a_h$ and $a_j=0$ for $j>h$, its \emph{pitch} $\pi(a,a_0)$ is the \emph{maximum} integer such that $a_0-\sum_{i=1}^{\pi(a,a_0)} a_i  \geq 0$.
\end{definition}
For example, the classical clique inequality $\sum_{i\in C} x_i \leq 1$, where $C$ is a clique, have pitch equal to one.

The following result for packing problems can be seen as the dual of Theorem~\ref{th:setcover} for set cover. It can be derived by using the so called ``Decomposition Theorem'' due to Karlin, Mathieu, and Nguyen \cite{karlin/ipco/2011}. Here we give a direct simple proof that follows the approach used throughout this paper.

\begin{lemma}\label{th:packing}
Consider any packing problem instance given by a matrix $A\in \R_+^{m\times n}$ and a vector $b\in \R_+^m$. Let $\pi=O(1)$ be a fixed positive integer. Then, $(\pi+1)$-$\sos$ satisfies all valid inequalities for $\F_{A,b}$ of pitch at most $\pi$.
\end{lemma}
\begin{proof}
Suppose $a_0 - a^\top x\geq 0$ is a valid inequality for $\F_{A,b}$ of pitch $\pi$ with $a_0,a\geq 0$. The claim follows from Proposition~\ref{th:pdrel} by showing that $a_0 - a^\top x$ admits a $(\pi+1)\text{-}\sos$ certificate.

Let $S:=\supp(a)$ and $x_I:=\prod_{i\in I}x_i$, for $I\subseteq [n]$. By~\eqref{Kr5} (choose $Z=I$), for any given $I\subseteq S$ we have $x_I(a_0 - a^\top x)\equiv x_I(a_0 - \sum_{i\in I} a_i- \sum_{i\not\in I} a_i x_i) \pmod {\IB}$.

Let $F:=\{I:I\subseteq S,(a_0 - \sum_{i\in I} a_i)<0\}$ and $T:=\{J:J\subseteq S, J\not \in F\}$ (and therefore if we set to 1 all the variables $x_i$ with $i\in I$, for $I\in F$, then the assumed valid inequality $a_0 - a^\top x\geq 0$ is violated).
 Let $V:=\{x\in \R^n:x_I=0 \  \forall I\in F , \  x_k^2-x_k=0 \ \forall k\in[n]\}$ and note that any feasible integral solution belongs to $V$.

For any given $\delta_J^S$, let $\bar{\delta}_J^S$ denote the ``truncated'' version of ${\delta}_J^S$ obtained from $\delta_J^S$ by zeroing all the monomials $x_I$ with $I\in F$ (observe that $\bar{\delta}_J^S=0$ for $J\in F$). Clearly, $\deg(\bar{\delta}_J^S)\leq \pi$, since $a_0 - a^\top x\geq 0$ has pitch at most $\pi$.
%
%
Note that $\sum_{I\subseteq S} \bar{\delta}_I^S=\sum_{I\in T} \bar{\delta}_I^S=1$,  $(\bar{\delta}_I^S)^2\equiv \bar{\delta}_I^S \pmod{\I(V)}$ and $\bar{\delta}_I^S(a_0 - a^\top x)\equiv \bar{\delta}_I^S(a_0 - \sum_{i\in I} a_i) \pmod{\I(V)}$. These can be derived by \emph{multilinearizing}, and by \emph{zeroing} all the monomials  from $\I(V)$ that are on the left and right-hand side of \eqref{Kr1}, \eqref{Kr2} and \eqref{Kr5}, respectively.

If follows that
\begin{align}\label{gigia}
  \bar{\delta}_I^S(a_0 - a^\top x)= \bar{\delta}_I^S(a_0 - \sum_{i\in I} a_i) + h_I(x) \quad \text{for some } h_I(x)\in \I(V).
\end{align}
As said before, the term $\bar{\delta}_I^S(a_0 - \sum_{i\in I} a_i)$, that is on the right-hand side of \eqref{gigia}, is obtained from the left-hand side of \eqref{gigia} by multilinearizing, so replacing each occurrence of $x_i^2$ with $x_i$, and by zeroing all the monomials  from $\I(V)$ that appear on the left-hand side. Note that these latter monomials have degree at most $\pi+1$, since
they derive from multiplying a degree $\pi$ polynomial $\bar{\delta}_I^S$ with a linear function.
Therefore $\deg(h_I(x))\leq \pi+1$.
Then
{
\begin{align}\label{eq:pippo}
a_0 - a^\top x &= \left(a_0 - a^\top x\right)
\overbrace{\left(\sum_{I\in T} \bar{\delta}_I^S\right)}^{=1}
\overset{\eqref{gigia}}{=} \sum_{I\in T}\overbrace{\left(a_0 - \sum_{i\in I}a_i\right)}^{\geq 0}\bar{\delta}_I^S + f(x),
\end{align}
}
for $f(x)=\sum_{I\in T}h_I(x)\in \I(V)$ with $\deg(f)\leq \pi+1$.

From the above equivalence we see that $a_0 - a^\top x$ can be written$\pmod{\I(V)}$ as a conical combination of polynomials from $\{\bar{\delta}_I^S:I\in T\}$ of degree at most $\pi$.
The claim follows by transforming the above congruence$\pmod{\I(V)}$~\eqref{eq:pippo} into a congruence$\pmod{\IB}$, while still using bounded degree polynomials.

Since $f(x)\in \I(V)$, by looking at the definition of $\I(V)$ note that every monomial in $f(x)$ belongs to $\I(V)$ as well.
Then $f(x)=\sum_{I\in U} f_I \cdot x_I$, for some $U\subseteq 2^{[n]}$ such that, for all $I\in U$, we have $x_I\in \I(V)$ and $f_I\in \R$ (and, as already observed, $\deg(x_I)\leq \pi+1$).

If $f_I\geq 0$ then $f_I \cdot x_I\equiv f_I \cdot (x_I)^2\pmod \IB$; otherwise (i.e. $f_I<0$), since $x_I\in \I(V)$. Then, for some $\lambda, \gamma\geq 0$, there is a valid constraint from a conical combination of valid constraints $c_I(x):=\left(\lambda^\top (b-Ax) + \gamma^\top x\right)\geq 0$
that is violated by setting $x_i=1$ for $i\in I$, i.e. $c(x_I)<0$. Therefore (recall $x_I\cdot c_I(x_I)\equiv x_I \cdot c_I(x)\pmod \IB$) $$f_I \cdot x_I \equiv \left(\sqrt{\frac{f_I}{c_I(x_I)}} x_I\right)^2 c_I(x)\pmod \IB.$$
It follows that $a_0 - a^\top x$ admits a $(\pi+1)\text{-}\sos$ certificate:
\begin{align}\label{pipp2}
  a_0 - a^\top x &\equiv s_0 + \sum_{i=1}^m s_i g_i \pmod \IB, \quad \text{ for some } s_i\in \Sigma_{\pi+1},
\end{align}
where $g_i\geq 0$, for $i\in[m]$, denotes the $i$-th constraint from $b-Ax\geq 0$ and
$\Sigma_{\pi+1}:=\{\sum_i q_i^2:q_i\in \R[x]_{\pi+1}\}$.
\end{proof}


Let $P=\{x\in \R^n:0\leq x_i\leq 1  \ \forall i\in [n],Ax\leq b\}$ denote the linear relaxation of $\F_{A,b}$.
For $t\in \N$, recall that $P^*$ and $P^{(t)}$ denote the integer hull and the $t$-th CG closure, respectively, of the starting linear program $P$, and $opt^{(t)}( c ):= \max\{c^\top x: x\in P^{(t)}\}$. Without loss of generality, we will assume that $c\in \R_+^n$ (otherwise it is always optimal to set $x_i=0$ whenever $c_i\leq 0$).
Let $Sol(d)$ denote the set of feasible solutions for $d\text{-}\sos$ projected to the original space of the variables. Let $opt_d(c):=\max\{c^\top x: x\in Sol(d)\}$.

The following result shows that fixed rank CG closures of packing problems can be approximated to any arbitrarily precision, and in polynomial time, by using the standard $\sos$ hierarchy.
\begin{theorem}
For any fixed $t\in \N$ and $\eps>0$, there is an integer $d=d(t,\eps)$ such that $opt_d(c)\leq (1+\eps) opt^{(t)}( c )$, for all $c\in \R_+^n$.
\end{theorem}
\begin{proof}
For any fixed $t\in \N$ and $\eps>0$, choose $d\in \N$ such that $(d/(d-1))^t \leq 1+\eps$.
Let $opt_d$ (or $opt^{(t)}$) denote $opt_d(c)$ (or $opt^{(t)}(c)$), for short.

If $opt_d \leq opt^{(t)}$ than we are done.
Otherwise ($opt_d > opt^{(t)}$),
let $x^{(t)}:=\phi_t \cdot x^*$ where $\phi_t:=(\frac{d-1}{d})^{t}$. It follows that $opt_d=c^\top x^*\leq (1+\eps) c^\top x^{(t)}$.
We show that $x^{(t)}$ is feasible for  the rank-$t$ CG closure. This imples that $c^\top x^{(t)}\leq opt^{(t)}$, and the claim follows since $opt_d\leq (1+\eps) c^\top x^{(t)}\leq (1+\eps)opt^{(t)}$.

The proof is by induction on $t$. As a base of induction note that when $t=0$ then clearly $x^{(0)}\in P=P^{(0)}$.

Assume now, by the induction hypothesis, that $x^{(t-1)}\in P^{(t-1)}$ for any rank equal to $(t-1)$ with $t\geq 1$. We need to show that it is valid also for rank $t$.
If the pitch of a generic rank-$t$ valid inequality for $P^{(t)}$ is at most~$d-1$, then by Lemma~\ref{th:packing} it follows that any feasible solution $x\in Sol(d)$ (and therefore $x^{(t)}$) satisfies this inequality.
Otherwise, consider a generic rank-$t$ valid inequality $\lfloor a_0\rfloor - a^\top x\geq 0$ of pitch larger than $d-1$, where
$a_0 - a^\top x\geq 0$ is any valid inequality from the closure $P^{(t-1)}$. By induction hypothesis note that $a_0 - a^\top x^{(t-1)}\geq 0$.
Since the pitch is higher than $d-1$ then $a_0>d-1$ (vector $a$ can be assumed, w.l.o.g., to be non-negative and integral) and therefore $\frac{a_0}{\lfloor a_0\rfloor} \leq \frac{d}{d-1}$ and by multiplying the solution $x^{(t-1)}\in P^{(t-1)}$ by $(d-1)/d$ we obtain a feasible solution for the rank-t CG closure.
\qd
\end{proof}
%
\section{Conclusions and Future Directions}\label{sect:conclusions}
A breakthrough result \cite{LeeRS15} of Lee, Raghavendra, and Steurer shows that the standard $\sos$ is ``\emph{optimal}'' for Constraint Satisfaction Problems among all semidefinite programs of comparable size.
In ~\cite{KurpiszLM17} and \cite{Kurpisz2017}, the standard $\sos$ is shown to be ``\emph{pessimal}'' for simple problems, meaning that it requires exponential size to get any bounded approximation.

The standard $\sos$ has been defined with respect to the standard monomial basis, which looks like a ``natural'' choice, but in fact it turns out to be an \emph{arbitrary} choice. This way can be ``good'' or ``bad'' depending on the problem at hand.

In this paper, we have shown a first example of $\sos$ equipped with a different basis,
that is useful in \emph{asymmetric} situations. The proposed approach overcomes some provable limitations of the standard $\sos$.

A very challenging open question is to understand what is the ``right'' basis for the problem that we want to address. Roughly speaking, can we transform the Recipe~\ref{recipe} into an effective algorithm? Any progress in this direction would be of considerable interest.
%
%
\begin{tcolorbox}[colframe=white, colback=white]
\emph{A mia mamma, che esiste per mancanza.}
\end{tcolorbox}


{
\bibliographystyle{abbrv}
\bibliography{ref}
}
%

\ifshort
\end{document}
\else

\appendix

\section{Sum of Squares Over the Boolean Hypercube}
\subsection{The Complexity of Computing $\Ss$-$\sos$ Certificates}\label{sect:qsosproofscomplexity}

\begin{lemma}\label{th:sdpsize}
Consider any given set of polynomials $\Ss\subseteq \R[\x]/\IB$ with $|\Ss|=n^{O(d)}$, for some $d\in\N$.
Then the existence of a $\Ss$-$\sos$ certificate can be decided by solving a semidefinite programming feasibility problem. The dimension of the matrix inequality is bounded by $n^{O(d)}$.
\end{lemma}

In the following we sketch the proof of the above lemma.
%
For simplicity, we sketch this for the case where the semialgebraic set $\F$ has no inequalities (so $m=0$ in \eqref{eq:semialgebraicset}).
The generalization to the case with inequalities follows in a similar vein (see e.g. Example~\ref{ex:sosauto}).

\paragraph{Testing if $f(x)$ is $\Ss$-$\sos$} We start recalling how to check if a polynomial $f(x)$ is $\Ss$-$\sos$, i.e check if (see Definition~\ref{def:polysos}): $$f(x)=\sum_{i=1}^r q_i(x)^2 \text{ for some }q_1,\ldots,q_r\in \Ss.$$
Note that it is ``$f(x)=\ldots$'' and not ``$f(x)\equiv \ldots \pmod{\IB}$''.
Then we show how to generalize this for checking whether $f(x)$ is $Q$-$\sos \pmod{\IB}$.

By overloading notation, let $x$ denote the vector of all monomials in $\R[\x]_n$ in a fixed order, say degree lexicographic. Recall that a polynomial $s(x)$ is a sum of squares if and only if there exists a positive semidefinite matrix $W$, denoted $W\succeq 0$, such that $s(x)= x^\top W x$. We review this in the following lemma.
\begin{lemma}\label{th:sospsd}
Let $s(x)\in \R[\x]$. The following statements are equivalent:
\begin{enumerate}
\item $s(x)$ has a representation as a sum of squares in $\R[\x]$.
\item There is a matrix $W$ such that $s(x) = x^\top W x$ with $W\succeq 0$, where $x$ denotes the vector of all monomials in $\R[\x]_n$.
\end{enumerate}
\end{lemma}
\begin{proof}
The matrix $W$ is PSD if and only if there is a factorization $W=V^\top V$. If this holds then
$s(x)=x^\top W x= x^\top V^\top V x = (V x)^\top (V x)=\sum_i \left((Vx)_i\right)^2$
is a $\sos$. Vice versa, if $s(x)=\sum_i \left((Vx)_i\right)^2$ then going backward in the previous equality the claim follows.
\qd
\end{proof}
By using the previous lemma it follows that $f(x)$ is a $\sos$ if and only if there is a symmetric matrix $W$ (known as the Gram matrix of the $\sos$ representation) that satisfies:
$s(x)=x^\top W x, \  W\succeq 0$.
Notice that the latter is a semidefinite program, since $f(x)=x^\top W x$ is affine in the matrix $W$, and thus the set of possible Gram matrices $W$ is given exactly by the intersection of an affine subspace and the cone of positive semidefinite matrices.

Consider any finite set of polynomials $\Ss\subseteq \R[\x]$ with $|\Ss|=n^{O(d)}$ and let $Q=\spns$ (for any positive constant $d$). Let $S$ be the matrix having as columns the spanning set $\Ss$. It follows that for any vector $q\in Q$ there is a vector $u\in \R^{|\Ss|}$ such that $q=S u$.

Since we want to check if $f(x)$ is $Q$-$\sos$ then we need to check if $f(x)=\sum_i \left((Vx)_i\right)^2$ and each $(Vx)_i$ belongs to $Q$ and therefore this happens if it exists a $u_i\in \R^{|\Ss|}$ such that $(S u_i)^\top x = (Vx)_i$. Let $U$ denote the matrix whose columns are the $u_i$, then we have the following:
$\sum_i \left((Vx)_i\right)^2 = x^\top S (UU^\top) S^\top x$.
%
Polynomials are expressed in the new basis $S^\top x$ (this basis is in general not isomorphic to the standard monomial basis of degree $d$) and the complexity is given by the size of the matrix $UU^\top$, i.e. $n^{O(d)}$.

\paragraph{Testing if $f(x)$ is $\Ss$-$\sos \pmod{\IB}$} The previous method can be adapted to check whether $f(x)$ is $\Ss$-$\sos \pmod{\IB}$. Actually, it is more general, it can be adapted to check whether $f(x)$ is $\Ss$-$\sos \pmod{\I}$, where $\I\subseteq \IB$ is any ideal for which we have the \GB basis $G$ (note that $\{x_i-x_i^2, i\in[n]\}$ is the \GB basis for $\IB$). We explain this in the following.

The vector $x$ can be replaced by the vector of all the different monomials after the division by $G$ (these are all the multilinear monomials if $\I=\IB$) since
$\R[\x]/\I$ is spanned by these monomials. This can decrease the size of the unknown matrix $W$, making the final SDP smaller than before. Setting up $W$ as a symmetric matrix of indeterminates we proceed as explained before (so $W= S (UU^\top) S^\top$). Let $s=x^\top W x$.
Let the normal forms of $f$ and $s$ with respect to a reduced \GB basis $G$ of $\I$ be $\bar{f}$ and $\bar{s}$, respectively (for the case $\I=\IB$, $\bar{f}$ and $\bar{s}$ are the multilinear representation of $f$ and $s$, respectively). Then since $f \equiv \bar{f} \pmod \I$ and $s \equiv \bar{s} \pmod \I$ and $\bar{f}$ and $\bar{s}$ are fully reduced with respect to $G$, we have that $f \equiv s \pmod \I$ if and only if $\bar{f} = \bar{s}$. Therefore, to check if $f(x)$ is $\Ss$-$\sos \pmod{\I}$, we equate the coefficients of $\bar{f}$ and $\bar{g}$ for like monomials and check whether the resulting linear system in the $W_{ij}$'s has a solution with $W \succeq 0$.
%

\begin{example}\label{ex:sosauto}
Consider the following set $\F=\{x\in \R^2: x_1-x_1^2=x_2-x_2^2=0, x_1+x_2-\eps\geq 0\}$ where $\eps\in(0,1)$. We want to show that the valid inequality $x_1+x_2-1\geq 0$ admits a $\Ss$-$\sos$ certificate, where $\Ss$ is the set of the elementary symmetric polynomials in 2 variables, i.e $\Ss=\{1,x_1+x_2,x_1x_2\}$ and therefore $\spns$ is the ring of all symmetric polynomials. Let $x=[1,x_1,x_2,x_1x_2]^\top$, the matrix $S$ is equal to

$ S = \begin{bmatrix}
       1 & 0 & 0 \\
       0 & 1 & 0 \\
       0 & 1 & 0 \\
       0 & 0 & 1
     \end{bmatrix}$ and the new basis is
$S^\top x = [1,x_1+x_2,x_1x_2]^\top$.

We want to show that $x_1+x_2-1\geq 0$ admits a $\Ss$-certificate, therefore we need to show that
\begin{align}\label{eq:qsosformexample}
 x_1+x_2-1\equiv s_0(x)+ s_1(x)(x_1+x_2-\eps) \pmod{\IB}.
\end{align}
where $s_0,s_1\in \{s\in\R[\x]: s=\sum_i q_i(x)^2, q_i\in \spns\}$. By Lemma~\ref{th:sospsd} there are two PSD matrices $W_0$ and $W_1$ such that $s_0(x)=x^\top W_0 x$ and $s_1(x)=x^\top W_1 x$ with the additional constraint on the structure of $W_0$ and $W_1$ given by the restriction that $q_i\in \spns$.
Let us first perform the change of basis $\sigma_i=(S^\top x)_i$ for $i=0,1,2$. So the new variables are $\sigma_0=1$, $\sigma_1=x_1+x_2$ and $\sigma_2=x_1x_2$ and the corresponding vector form $\sigma=[1,\sigma_1,\sigma_2]^\top$. Note that in the new basis $\sigma_1^2\equiv \sigma_1+2\sigma_2 \pmod{\IB}$, $\sigma_2^2\equiv \sigma_2 \pmod{\IB}$ and $\sigma_1 \sigma_2 \equiv 2\sigma_2 \pmod{\IB}$ which correspond to the multilinear forms in the new basis.
By rephrasing our goal in the new basis, we need to show that \begin{align}\label{eq:qsosformexample2}
 \sigma_1-1\equiv \sigma^\top T_0 \sigma+ (\sigma^\top T_1 \sigma)(\sigma_1-\eps) \pmod{\IB},
\end{align}
for some PSD matrices $T_0,T_1$ with
$ T_i = \begin{bmatrix}
       t_{i00} & t_{i01} & t_{i02} \\
       t_{i01} & t_{i11} & t_{i12} \\
       t_{i02} & t_{i12} & t_{i22}
     \end{bmatrix}$
for $i=0,1$.
By writing~\eqref{eq:qsosformexample2} in the multilinear form, our goal is to prove that there are two PSD matrices $T_0,T_1$ such that the following is satisfied:
{\small{
\begin{align*}
&\sigma_1-1= \underbrace{t_{000}-\eps t_{100}}_{\alpha} + \underbrace{(t_{011}+2 t_{001}+ t_{100} + (t_{111} + 2 t_{101})(1- \eps) )}_{\beta} \sigma_1 +\\ &\underbrace{(2t_{011}+ t_{022} +2t_{002}+4 t_{012}+ 6 t_{111} + 4 t_{101}+2t_{122}+ 4t_{102}+8t_{112}-\eps(2t_{111}+t_{122}+2t_{102}+4t_{112})}_{\gamma} \sigma_2
\end{align*}
}}
So the solution of the following SDP$=\{\alpha =-1, \beta=1, \gamma=0, T_0\succeq 0, T_1\succeq 0\}$ gives the desired $\Ss$-$\sos$ certificate.
By choosing $T_0=[0,0,1]^{\top}[0,0,1]$ and $T_1=\frac{1}{\eps}[1,-1,1]^{\top}[1,-1,1]$ the SDP is satisfied.

\end{example}


 \section{$k$-Small Minimal Covers}\label{sect:smallmincovers}
Consider a generic inequality of any given integer problem as written in the covering form, i.e. inequality $g(x)=a^\top x-a_0\geq 0$ with $a\geq 0$ (here, abusing notation, every variable $x_j$ is either the original one or its negation $1-x_j$). For each such constraint let $V_a=\supp(a)$ be the set of variables in this constraint. Add to the $\Ss_{A}(k)$-$\sos$ polynomials the set of all $C$-symmetric polynomials with $C\subseteq V_a$ and $|C|\leq k$. Consider any valid $k$-small minimal cover of type $\sum_{i\in C} x_i\geq 1$, with $|C|\leq k$ (the other cases are similar). 
We sketch that it admits an $\Ss_{A}(k)$-$\sos$ certificate:
{\small
\begin{align*}
&\sum_{i\in C} x_i- 1
= \left(\sum_{i\in C} x_i- 1\right)\overbrace{\left(\sum_{i=0}^{|C|} \underbrace{\sum_{I\subseteq C:|I|=i} \delta^{C}_I}_{symmetric}\right)}^{=1}\equiv
\left(\sum_{i=0}^{|C|} (i- 1) \underbrace{\sum_{I\subseteq C:|I|=i} \delta^{C}_I}_{symmetric}\right)\\
&\equiv \overbrace{\left(\frac{-1}{\sum_{i\in V_a\setminus C} a_i-a_0}\right)}^{> 0}\delta^{C}_0 \overbrace{\left(\sum_{i\in V_a\setminus C} a_i + \sum_{i\in C} a_i x_i-a_0\right)}^{<0}+\left(\sum_{i=1}^{|C|} (i- 1) \underbrace{\sum_{I\subseteq C:|I|=i} \delta^{C}_I}_{symmetric}\right)\\
&\equiv \overbrace{\left(\frac{-1}{\sum_{i\in V_a\setminus C} a_i-a_0}\right)}^{> 0}\delta^{C}_0 \left(\sum_{i\in V_a} a_ix_i-a_0 +\sum_{i\in V_a\setminus C}a_i(1-x_i)\right)+\left(\sum_{i=1}^{|C|} (i- 1) \underbrace{\sum_{I\subseteq C:|I|=i} \delta^{C}_I}_{symmetric}\right)\\
&\equiv
\underbrace{\left(\sqrt{\frac{-1}{\sum_{i\in V_a\setminus C} a_i-a_0}} \delta^{C}_0\right)^2}_{s(x)} \underbrace{\left(\sum_{i\in V_a} a_ix_i-a_0\right)}_{g(x)} +  \sum_{i\in V_a\setminus C}\underbrace{\left(\sqrt{\frac{-a_i}{\sum_{i\in V_a\setminus C} a_i-a_0}} \delta^{C}_0\right)^2}_{s_i(x)}\underbrace{(1-x_i)}_{g_i(x)}+\\ &+\underbrace{\left(\sum_{i=1}^{|C|} \sqrt{i- 1 }{\sum_{I\subseteq C:|I|=i} \delta^{C}_I}\right)^2}_{s_0(x)}
\end{align*}
}

\subsection{An Application}
As  in~\cite{BienstockZ04,zuckerberg2004set}, Theorem~\ref{th:setcover} can be generalized to handle 0/1 integer problems with non-negative constraints having pitch bounded by a constant $p$. More precisely, consider the feasible region for the $0$-$1$  problem defined by $A$:
\begin{align}\label{setcov2}
\F_A=\{x\in \{0,1\}^n:Ax\geq b\}
\end{align}
where $b\in \R^{m}_+$ and each constraint in $Ax\geq b$ has pitch at most $p$. (For example any inequality $a^\top x-a_0\geq 0$ with non-negative integral coefficients $a_i\in \{0,1,\ldots,p\}$ has pitch at most $p$.) In this case the number of minimal covers is polynomially bounded. Since the integral polytope defined by using the minimal covers and the integrality constraints coincides with~\eqref{setcov2} (see e.g. \cite{Conforti:2014}), then we can extend Theorem~\ref{th:setcover} to this more general case.


\section{Omitted proofs}
\subsection{Proof of Theorem \ref{th:maxsymconstr}}\label{sect:toyd}

Before proving the bound given in Theorem~\ref{th:maxsymconstr} on the number of levels for our simple example we need some preliminaries. In particular we first introduce the $\sos$ hierarchy in matrix form that is more convenient for proving lower bounds. In the following we assume that the $\sos$ hierarchy is the ``standard'' one, namely the one that follows by considering the subspace of bounded degree polynomials as functional basis.
\subsubsection{The Sum-of-Squares Hierarchy in Matrix Form}
Consider the $\sos$ hierarchy for approximating the convex hull of the semialgebraic set
\begin{equation} \label{eq:definition_of_set_K}
P = \set{x \in \set{0,1}^n~:~g_{\ell}(x)\geq 0, \forall \ell\in [p]}
\end{equation}
where $g_\ell(x)$ are linear constraints and $p$ a positive integer. The form of the $\sos$ hierarchy we use here is equivalent to the one introduced before and follows from applying a change of basis to the dual certificate of the refutation of the proof system (see~\cite{KLMipco16} for the details on the change of basis). We use this change of basis  in order to obtain a useful decomposition of the moment matrices as a sum of rank one matrices of a special kind.

For any $I\subseteq N=[n]$, let $x_I$ denote the $0/1$ solution obtained by setting $x_i = 1$ for $i \in I$, and $x_i = 0$ for $i\in N\setminus I$. For a function $f:\set{0,1}^n \rightarrow \R$, we denote by $f(x_I)$ the value of the function evaluated at $x_I$. In the $\sos$ hierarchy defined below there is a variable $\y_I$ that can be interpreted as the ``relaxed'' indicator variable for the solution $x_I$. We point out that in this formulation of the hierarchy the number of variables $\{\y_I:I\subseteq N\}$ is exponential in $n$, but this is not a problem in our context since we are interested in proving lower and upper bounds rather than solving an optimization problem.

Let $\PS_t(N)$ be the collection of subsets of $N$ of size at most $t\in\mathbb{N}$. For every $I \subseteq N$, the $q$-\text{zeta vector}  $Z_I \in \mathbb{R}^{\PS_q(N)}$ is a $0/1$ vector with $J$-th entry ($|J| \leq q$) equal to $1$ if and only if $J \subseteq I$.\footnote{In order to keep the notation simple, we do not emphasize the parameter $q$ as the dimension of the vectors should be clear from the context.}
Note that $Z_IZ_I^\top$ is a rank one matrix and the matrices considered in Definition~\ref{def:sos_definition_integer_hull} are linear combinations of these rank one matrices.

\begin{definition}~\label{def:sos_definition_integer_hull}
 The $t$-th round SoS hierarchy relaxation for the set $P$ as given in~\eqref{eq:definition_of_set_K}, denoted by $\SoS_t(P)$, is the set of variables $\{\y_I \in \mathbb{R}:  \forall I \subseteq N\}$ that satisfy
 \begin{eqnarray}
 \sum_{\substack{I \subseteq N}} \y_I&=& 1, \label{eq:sos_sum_condition_ih} \\
 \sum_{\substack{I \subseteq N}} \y_I Z_I Z_I^\top &\succeq& 0, \text{ where } Z_I \in \mathbb{R}^{\PS_{t+1}(N)} \label{eq:sos_variables_ih}\\
 \sum_{\substack{I \subseteq N}} g_\ell(x_I)\y_I Z_IZ_I^\top &\succeq& 0, ~ \forall \ell\in [p]  \text{, where }  Z_I \in \mathbb{R}^{\PS_{t}(N)}  \label{eq:sos_constraints_ih}
\end{eqnarray}
\end{definition}
It is straightforward to see that the SoS hierarchy formulation given in Definition~\ref{def:sos_definition_integer_hull} is a relaxation of the integral polytope. Indeed consider any feasible integral solution $x_I \in P$ and set $\y_I=1$ and the other variables to zero. This solution clearly satisfies~\eqref{eq:sos_sum_condition_ih} and~\eqref{eq:sos_variables_ih} because the rank one matrix $Z_IZ_I^\top$ is positive semidefinite (PSD), and~\eqref{eq:sos_constraints_ih} since $x_I\in P$.

For a set $Q \subseteq [0,1]^n$, we define the projection from $\SoS_t(Q)$ to $\mathbb{R}^n$ as $x_i = \sum_{i \in I \subseteq N} y_I^N$ for each $i \in \set{1,...,n}$. The \emph{SoS rank} of $Q$, $\rho(Q)$, is the smallest $t$ such that $\SoS_t(Q)$ projects exactly to the convex hull of $Q \cap \set{0,1}^n$.

\subsubsection{Using Symmetry to Simplify the PSDness Conditions} \label{sect:main_theorem}

In this section we present a theorem given in~\cite{KLMipco16} that can be used to simplify the PSDness conditions~\eqref{eq:sos_variables_ih} and~\eqref{eq:sos_constraints_ih} when the problem formulation is very symmetric. More precisely, the theorem can be applied whenever the solutions and constraints are symmetric in the sense that $w_I^N = w^N_J$ whenever $|I|=|J|$ where $w_I^N$ is understood to denote either $y^N_I$ or $g_\ell(x_I)y^N_I$. In what follows we denote by $\mathbb{R}[x]$ the ring of polynomials with real coefficients and by $\mathbb{R}[x]_d$ the polynomials in $\mathbb{R}[x]$ with degree less or equal to $d$.

\begin{theorem}[\cite{KLMipco16}] \label{thm:PSD_as_polynomial}
For any $t\in \{1,\ldots,n\}$, let $\mathcal{S}_t$ be the set of univariate polynomials $G_h(k)\in \R[k]$, for $h \in \{0,\ldots,t\}$, that satisfy the following conditions:
\begin{align}
G_h(k)&\in \R[k]_{2t} \label{eq:degree}\\
G_h(k)&=0 \qquad \text{for }k\in \{0,\ldots,h-1\}\cup \{n-h+1,\ldots,n\} \text{, when }h\geq 1\label{eq:zeros}\\
G_h(k)&\geq0 \qquad \text{for }k\in [h-1, n-h+1]\label{eq:geq0}
\end{align}
For any fixed set of values $\{w^N_k \in \R:k =0,
\ldots,n\}$,
if the following holds
\begin{align}
\sum_{k=h}^{n-h}\binom{n}{k} w^N_k G_h(k) &\geq 0 \qquad \forall G_h(k)\in \mathcal{S}_t \label{eq:sym_psd_cond}
\end{align}
then
$$
 \sum_{k = 0}^n w^N_k \sum_{\substack{I \subseteq N \\ |I| = k}} Z_IZ_I^\top \succeq 0 \qquad
$$
where $Z_I \in \mathbb{R}^{\PS_{t}(N)}$.
\end{theorem}
Note that polynomial $G_h(k)$ in~\eqref{eq:geq0} is non-negative in a \emph{real interval}, and in~\eqref{eq:zeros} it is zero over a \emph{set of integers}. Moreover, constraints \eqref{eq:sym_psd_cond} are trivially satisfied for $h> \lfloor n/2 \rfloor$.

\subsubsection{The Simple Example Proof}
The single constraint of the simple example can be rewritten, w.l.o.g., as follows:
$$g(x)=\sum_{i=1}^n x_i - L+1-\frac{1}{P}\geq 0$$
where $L$ and $P$ are positive integers. Clearly any integral $\{0,1\}$-solution requires to set to one at least $L$ variables.

Let $(LP)$ be the polytope $\left\{x \in [0,1]^n : \\ g(x) \geq 0 \right\}$. The $\SoS$ \emph{rank} is the minimal number of levels needed to obtain the integer hull $(IP)$ of $(LP)$.

In the following we will restrict the analysis to $L\leq\lceil n/2 \rceil$. Consider any solution that satisfies the following conditions:
\begin{equation}\label{sol}
\left\{
\begin{array}{ll}
y_k^N =0 \qquad \text{for } k\leq L-2\\
y_k^N >0 \qquad \text{for } k\geq L-1\\
\sum_{k=0}^n y_k^N \binom{n}{k}=1
\end{array}
\right.
\end{equation}
Note that in~\eqref{sol} we do not impose any restriction on the exact value of the positive probabilities. 
The value of the suggested solution is
$\sum_{k=L-1}^n \binom{n}{k}\y_k k$.
By choosing $P$ sufficiently large we will show that almost all the probability mass (but an arbitrarily small part) can be assigned to $\y_{L-1}$, resulting therefore into an objective function value equal to $L-1+\eps$, (for any $\eps>0$) and an integrality gap of $\frac{L}{L-1+\eps}$.

\begin{lemma}\label{maxsymconstr}
For $L\leq\lceil{n/2}\rceil$ and a suitable large value of $P$ that depends on $n$ the $\SoS$ rank for $(LP)$ is at least $n-L+1$.
\end{lemma}
\begin{proof}

For any solution that satisfies~\eqref{sol} there is a unique nonpositive term in conditions \eqref{eq:sym_psd_cond}, namely $z_{L-1}^N G_h(L-1)=y_{L-1}^N(-1/P) G_h(L-1)=-\eps G_h(L-1)$ (for some $\eps=y_{L-1}^N/P>0$), where we use the following notation $z_{k}^N=y_k^N g(k)$ (with $g(k)$ denoting the value of the constraint $g(x)$ when exactly $k$ variables are set to one).

If we chose $h$ such that $L-1=n-h$ then we would have that $z_k^N G_h(k)$ is equal to zero for all $k\not = n-h$, and by choosing $G_h(k)$ such that $G_h(L-1)>0$ 
we would have that \eqref{eq:sym_psd_cond} is never satisfied. To avoid this problem we assume that $L-1\leq n-h-1$ and since $h\leq \lfloor{n/2}\rfloor$, the claim holds when $L\leq n-\lfloor{n/2}\rfloor=\lceil{n/2}\rfloor$.

According to Theorem~\ref{thm:PSD_as_polynomial} and~\eqref{sol} note that
\begin{itemize}
\item $G_h(k)$ has $2t$ roots.
\item $G_h(k)$ has at least $h-1+1+n-(n-h+1)+1=2h$ roots outside the (open) interval $(h-1,\ldots,n-h+1)$.
\item $G_h(k)$ has at most $2(t-h)$ roots within the (open) interval $(h-1,\ldots,n-h+1)$. Moreover $G_h(k)\geq 0$ for any $k\in(h-1,\ldots,n-h+1)$ and therefore the at most $2(t-h)$ roots that are within the (open) interval $(h-1,\ldots,n-h+1)$ must appear in pairs. It follows that $G_h(k)$ has at most $t-h$ different roots within the (open) interval $(h-1,\ldots,n-h+1)$.
\end{itemize}
Consider any $h$ such that $h\leq L-1\leq n-h-1$ (if $L-1\leq h-1$ then \eqref{eq:sym_psd_cond} is trivially satisfied).
Note that there are $n-h-L+1$ terms $z^N_k>0$ for $k\in \{L,\ldots,n-h\}$ (note that $L\leq n-h$ by assumption, so set $\{L,\ldots,n-h\}$ is never empty).
From the above arguments we know that $G_h(k)$ has at most $t-h$ different roots within the (open) interval $(h-1,\ldots,n-h+1)$.
So if $t-h$ is strictly smaller than the number $n-h-L+1$ of terms $z^N_k>0$ (with $k\in \{L,\ldots,n-h\}$) then it exists a $k^*\in \{L,\ldots,n-h\}$ that is not a root for $G_h(k)$ and such that $z^N_{k^*} \binom{n}{k^*} G_h(k^*)>0$ (recall that $G_h(k)\geq 0$ within the considered interval which implies that $G_h(k^*)>0$). The latter condition is satisfied when  $t-h\leq n-h-L$, namely when $t\leq n-L$.
It follows that if $t\leq n-L$ then there exists a $k^*\in \{L,\ldots,n-h\}$ such that $z^N_{k^*} \binom{n}{k^*} G_h(k^*)>0$. 
Moreover, let $r_1,\ldots,r_{2t}$ be the roots of $G_h(x)$. Then $k^*\in \{L,\ldots,n-h\}$ can be chosen such that the following two conditions are both satisfied:
\begin{align}
&|k^*-r_i|\geq 1/2, \quad \text{ for every } i\in[2t], \label{rootdist}\\
&z^N_{k^*} \binom{n}{k^*} G_h(k^*)>0. \label{poscond}
\end{align}

Let $j^*$ such that $k^*=L-1+j^*$, where $j^*\in\{1,\ldots,n-h-L+1\}$. The claim follows by showing how to choose $P$ such that:
$$z^N_{L-1+j^*} \binom{n}{L-1+j^*} G_h(L-1+j^*)>\frac{\y_{L-1}}{P}\binom{n}{L-1}G_h(L-1).$$
From~\eqref{poscond} the above condition is equivalent to satisfy the following
\begin{align}\label{tempcond}
z^N_{L-1+j^*} >\frac{\y_{L-1}}{P}\frac{\binom{n}{L-1}}{\binom{n}{L-1+j^*}}\frac{G_h(L-1)}{G_h(L-1+j^*)}.
\end{align}
Clearly, the interesting cases are when $G_h(L-1)>0$. By the latter, \eqref{rootdist} and \eqref{poscond}, we have that:
\begin{align}\label{factUB}
\frac{G_h(L-1)}{G_h(L-1+j^*)}=\prod_{i=1}^{2t}\frac{ |L-1-r_i|}{ |L-1+j^*-r_i|}
\leq \prod_{i=1}^{2t}\left ( 1+ \frac{ j^*}{ |L-1+j^*-r_i|}\right )
\leq \prod_{i=1}^{2t}\left ( 1+ 2j^*\right ).
\end{align}
By~\eqref{factUB}, if the following is satisfied then \eqref{tempcond} holds.
\begin{align}\label{suffcond}
z^N_{L-1+j^*} >\frac{\y_{L-1}}{P}\frac{\binom{n}{L-1}}{\binom{n}{L-1+j^*}}\left ( 1+ 2j^*\right )^{2t}.
\end{align}
Then it is sufficient to choose $P$ such that
\begin{equation*}
P  \geq   2\frac{y_{L-1}^N}{y_{L-1+j^*}^N} \frac{\binom{n}{L-1}}{\binom{n}{L-1+j^*}} \frac{(1+2j^*)^{2t}}{j^*}.
\label{eq:knapsack_condition_in_poly_form}
\end{equation*}
Note that the right-hand side of the above inequality is bounded by a function of $n$.
\end{proof}

\section{On a very recent claim by Fiorini et al. \cite{FioriniHW17}}

We describe the approach suggested in \cite{FioriniHW17} for the 0/1 set cover problem which is also the main application advertised in the abstract. We observe in the following that their approach is essentially based on similar arguments as in this paper (formerly appeared in \cite{Mastrolilli17}) but specialized for a weaker framework that does not generalize to packing problems (see Section~\ref{sect:cgpack}).
We sketch this for pitch 2 in the following. The generalization to any pitch is straightforward.

Let $A$ be the $m\times n$ set cover matrix defined as in \eqref{setcov} and let $A_{ij}$ denote the $(i,j)$-entry of $A$. By overloading notation, we will interchangeably use $A_i$ to denote the $i$-th row of $A$ and its support. In \cite{FioriniHW17}, they consider the canonical monotone formula for set cover:
\begin{align}\label{eq:phiformula}
\phi:=\bigwedge_{i=1}^m \bigvee_{A_{ij=1}}x_j.
\end{align}
Starting with any convex set $Q\subseteq [0,1]^n$ containing $\F_A$ (see \eqref{setcov}) the improved relaxation is obtained  by recursively ``feeding'' $Q$ into the formula $\phi$, denoted by $\phi(Q)$ and defined as follows:
\begin{align}\label{eq:phi_Q_formula}
\phi(Q):=\bigcap_{i=1}^m \conv \left(\bigcup_{A_{ij=1}}\{x\in Q: x_j=1\}\right).
\end{align}
By starting with $Q:=[0,1]^n$ it is easy to see that $\phi([0,1]^n)=\{x\in [0,1]^n:Ax\geq e\}$.
This is also the base of induction in the proof of Lemma~\ref{th:bzcore} in this paper.
 So their approach obtains, after the first application, the starting linear program relaxation that corresponds to all pitch one inequalities (also used in~\eqref{eq:sc3}). Now let $Q:=\phi([0,1]^n)$ and let's analyze the second application, namely $\phi(Q)=\phi^2([0,1]^n)$:
\begin{align}\label{eq:phi_Q_formula2}
\phi(Q):=\bigcap_{i=1}^m \conv \underbrace{\left(\bigcup_{A_{ij=1}}\{x\in [0,1]^n:Ax\geq e, x_j=1\}\right)}_{U_i}.
\end{align}

It can be easily observed that the relaxation given by \eqref{eq:phi_Q_formula2} is obtained by considering the ``interaction'' of the $i$-th pitch 1 constraint (for any $i\in [m]$, see the outer intersection) with any other constraint $h\in [m]$ from $Ax\geq e$. The ``interaction'' is given by the common variables, denoted by $A_i\cap A_h$ in this paper, otherwise (i.e. $j\not\in A_h$) setting $x_j=1$ does not effect the corresponding constraint $A_h x\geq 1$. These are exactly the variables considered in $V_C$ with $C=\{i,h\}$.

Lemma~\ref{th:bzcore} gives a property of these interactions that are used for proving that these pairs of interactions are sufficient to show pitch 2 inequalities. Higher pitches use recursive polynomials which correspond to recursive application of $\phi$  by considering triplets for pitch 3 and so on, as in this paper.

\fi
\end{document}